\documentclass[11pt,a4paper]{amsart}

\usepackage{graphicx}
\newtheorem{theorem}{Theorem}
\newtheorem{lemma}[theorem]{Lemma}
\newtheorem*{lemma*}{Lemma}
\newtheorem{proposition}[theorem]{Proposition}
\newtheorem{corollary}[theorem]{Corollary}

\usepackage[utf8]{inputenc}
\usepackage[english]{babel}
\usepackage{latexsym}
\usepackage{amssymb}
\usepackage{amsmath}

\title{Critical Mandelbrot Cascades}

\author[J. Barral]{Julien Barral}
\address{LAGA (UMR 7539), D\'epartement de Math\'ematiques, Institut Galil\'ee, Universit\'e Paris 13, 99 avenue Jean-Baptiste
Cl\'ement , 93430 Villetaneuse, France}
\email{barral@math.univ-paris13.fr}

\author[A. Kupiainen]{Antti Kupiainen$^{1,2}$}
\address{University of Helsinki, Department of Mathematics and Statistics,
         P.O. Box 68 , FIN-00014 University of Helsinki, Finland}
\email{antti.kupiainen@helsinki.fi}

\author[M. Nikula]{Miika Nikula$^1$}
\address{University of Helsinki, Department of Mathematics and Statistics,
         P.O. Box 68 , FIN-00014 University of Helsinki, Finland}
\email{miika.nikula@helsinki.fi}

\author[E. Saksman]{Eero Saksman$^1$}
\address{University of Helsinki, Department of Mathematics and Statistics,
         P.O. Box 68 , FIN-00014 University of Helsinki, Finland}
\email{eero.saksman@helsinki.fi}

\author[C. Webb]{Christian Webb$^1$}
\address{University of Helsinki, Department of Mathematics and Statistics,
         P.O. Box 68 , FIN-00014 University of Helsinki, Finland}
\email{christian.webb@helsinki.fi}

\date{\today}

\newcommand{\N}{\mathbb{N}}

\newcommand{\R}{\mathbb{R}}

\newcommand{\E}{\mathbb{E}}

\newcommand{\Prob}{\mathbb{P}}

\newcommand{\ind}{\mathbf{1}}

\newcommand{\Acal}{\mathcal{A}}

\newcommand{\Bcal}{\mathcal{B}}

\renewcommand{\d}{\, \mathrm{d}}

\newcommand{\eqlaw}{\stackrel{d}{=}}

\keywords{Multiplicative cascades, critical temperature, KPZ formula, multifractal analysis}
\subjclass[2010]{Primary 60G57,28A78; Secondary 60G18,83C45,60G51}

\begin{document}
\maketitle

\begin{abstract}
We study Mandelbrot's multiplicative cascade measures at the  critical temperature.
As has been recently shown by 
 Barral, Rhodes and Vargas (\cite{barhova12}), an appropriately normalized sequence of cascade measures converges weakly in probability to a nontrivial limit measure. We prove that these limit measures have no atoms  and give bounds for the modulus of continuity of the cumulative distribution function of the measure. Using the earlier work of Barral and Seuret (\cite{BaSe07}), we compute the multifractal spectrum of the measures. We also extend the result of Benjamini and Schramm (\cite{BS09}), in which the KPZ formula from quantum gravity is validated for the high
temperature cascade measures, to the critical and low temperature cases.
\end{abstract}

\section{Introduction }

Random multiplicative cascade measures were introduced by B. Mandelbrot
\cite{mandel2},\cite{mandelbrot},\cite{mandel1}, as simple models exhibiting fractal and statistical features analogous to those observed experimentally in velocity fluctuations of fully developed turbulence. Since then
these multifractal measures have found applications in various fields ranging from
mathematical finance to disordered systems and two dimensional quantum
gravity (see \cite{bajinrhovar12} for references). In the field of disordered systems the (normalized) cascade measures
can be seen as  Gibbs measures of Generalized Random Energy Models with infinitely many levels or continuous hierarchies (see e.g. \cite{bov,bk}) or alternatively as   Gibbs measures of a model of a directed polymer on a
disordered  tree  \cite{DerSpo}.
 The mathematical study of
multiplicative cascades was initiated by Kahane \cite{Kahane} and Peyri\`ere \cite{pey74}
 and has since then been pursued by numerous people in analysis, probability and
mathematical physics, often independently (see \cite{barman1},\cite{barman2},\cite{barman3},\cite{we11} for references to some of the work).

\footnotetext[1]{Supported by
the  $\hbox{}^1$Academy of Finland and  $\hbox{}^2$ERC}

The simplest cascade measures are random measures on the unit interval
defined in terms of two inputs (see below): a real valued random variable $\xi$
(describing fluctuations at a fixed scale) and (inverse) temperature parameter $\beta>0$.
The behavior of the measures is rather insensitive to $\xi$ but depends strongly
on $\beta$. Derrida and Spohn \cite{DerSpo} argued in 1988
 that they exhibit a phase transition at
a critical value  $\beta_c$ of $\beta$ to a ``glassy" low temperature phase
in $\beta>\beta_c$. In the high temperature region $\beta<\beta_c$ the measures are
continuous (but singular with respect to the Lebesgue measure), already proven by Kahane and
Peyri\`ere  \cite{KP} in 1976. Progress in the critical $\beta=\beta_c$ case and the supercritical
$\beta>\beta_c$ cases has been slower to come. The reason is that whereas in the
subcritical case the cascade measures can be proven to exist as  non-degenerate limits of
positive martingales, in the critical and supercritical cases the martingale limit vanishes, and there is no obvious
candidate for a normalization leading to convergence in law to a non trivial limit. Very
recently,
 A\"{i}d\'{e}kon and Shi \cite{aishi11}  proved detailed asymptotics
for the probability distribution of the total mass of the cascade measures  in the critical case.
In the case where $\xi$ is Gaussian the fifth author  \cite{we11} obtained independently similar results, both in the critical and in the supercritical case, basing his approach  on the seminal paper by Bramson
\cite{br83}. Madaule \cite{ma11}  treated the supercritical case for general $\xi$. These results allow one to find the required renormalizations and to construct the limits
for the total mass (partition function).  Recently, this was extended to the measures themselves
by Barral, Rhodes and Vargas in \cite{barhova12}.
 These latter authors
also proved that the cascade measures are a.s. purely atomic in the supercritical case.

In this paper we study the cascade measures at the critical point. We give a simple proof
that  they have a.s. no
atoms based on a recent result by Buraczewski in \cite{bu09} and the aforementioned
results on the renormalization factors. We also give bounds for the modulus of continuity
of the cumulative distribution function of the critical measure which is of interest
for the attempts to use these measures as inputs for construction of random
plane curves by conformal welding (see \cite{AJKS}, \cite{she} for such constructions
in the high temperature case when the cascade measure is replaced by
exponential of the Gaussian Free Field, related to a continuous cascade model). In passing we note that our approach can also be used to  improve the known bounds for the modulus of continuity of the Mandelbrot measures in the subcritical case.
Next, we discuss the KPZ formula \cite{KPZ} of two dimensional quantum gravity in
the cascade context. The KPZ formula was 
reformulated  by Duplantier and Sheffield \cite{DS08} as a relation between Hausdorff dimensions of sets computed with the
Lebesgue measure and a random measure given by exponential of the Gaussian Free Field
and proven by them to be valid in the high temperature region. In the cascade context
the high temperature result was proved by  Benjamini and Schramm (\cite{BS09})  and
we show how their proof 
 generalizes to the critical and
low temperature cases. Finally, using the earlier work of Barral and Seuret (\cite{BaSe07}), we compute the multifractal spectrum of the measures in the critical and
low temperature cases.

It remains a challenge to extend the results of this paper and  \cite{barhova12} to
the stationary log-normal multiplicative chaos of Mandelbrot \cite{mandel1}
and the related
measures given as exponentials of the Gaussian Free Field (GFF) \cite{DS08}.
%Some progress was recently obtained in \cite{drsv}  in the convergence of
%the derivative martingale, an object that is believed to share the properties of
%the critical  measures. In \cite{drsv} it is shown that in the critical case for GFF there are no atoms, but their methods do not give as fine control of the continuity of the measures as ours.
Some progress to this goal has been  obtained  very recently \cite{drsv},\cite{drsv2},\cite{bknsw}. Especially, in \cite{drsv} it is shown that in the critical case for GFF there are no atoms, but their methods do not give as fine control of the continuity of the measures as ours. Moreover, \cite{bknsw} establishes partial counterparts of certain  results of the present paper for critical  GFF.

\section{Definitions and Results }\label{se:defres}

For simplicity,  in this paper we will consider  only multiplicative cascade measures on binary trees. We define the symbolic space as $\Sigma = \bigcup_{n=1}^\infty \{0,1\}^n$ and for convenience denote the $n$-th level by $\Sigma_n = \{0,1\}^n$ i.e. this set indexes the edges  of the tree on  $n$-th level. Let $\xi$ be a random variable such that
\begin{equation}\label{eq:criticalnormalization}
\E e^\xi = \frac{1}{2} \quad {\rm and} \quad \E \, \xi e^\xi = 0
\end{equation}
and
\begin{equation}\label{eq:momentconditions}
\E e^{(1+h)\xi} < \infty\quad  \textrm{for some } h > 0.
\end{equation}
The conditions \eqref{eq:criticalnormalization} are essentially a normalization that is convenient for studying the critical case ($\beta_c=1$ with this normalization)  and can be changed by considering instead $a \xi + b$ for $a,b \in \R$. 
E.g. in the Gaussian case $\xi\sim N(-2\log 2,2\log 2)$ satisfies \eqref{eq:criticalnormalization}.

The condition \eqref{eq:momentconditions} on the tail behavior of $\xi$ is a technical assumption that is required for the proofs of many of the results we are building upon. It is obviously satisfied in the Gaussian case.  Many of our results remain true on less stringent assumptions, and in some cases  we indicate this explicitly. 

To define the cascade measures, let $\{\xi_\sigma\}_{\sigma \in \Sigma}$ be an independent family of copies of $\xi$ and associate to every $\sigma = \sigma_1 \sigma_2 \dots \sigma_n \in \Sigma$ the sum
$$
X_\sigma = \xi_{\sigma_1} + \xi_{\sigma_1 \sigma_2} + \cdots + \xi_{\sigma_1 \sigma_2
\dots \sigma_n}.
$$
For any $\beta > 0$, consider the partition function
\begin{equation}
\label{eq:defZbeta}
Z_{\beta,n} = \sum_{\sigma \in \Sigma_n} e^{\beta X_\sigma} \quad \textrm{for } n = 1, 2, \dots.
\end{equation}
In other words, we consider a basic model of  the branching random walk
with $X_\sigma$ the positions of the $2^n$ particles at time $n$.
Interpreting $\sigma \in \Sigma_n$ as a spin configuration on $\{1,\dots,n\}$ we
recognize that $Z_{\beta,n}$ is the partition function of
a Generalized Random Energy Model with continuous hierarchies, see the end of this Section
for further discussion. % - with the energies $E_\sigma=-X_\sigma$, 
Finally, one may also view $Z_{\beta,n}$ as  the partition function of a model for a polymer on a tree
\cite{DerSpo}.

It is a classical result of Kahane and Peyri\`{e}re \cite{KP} that for $\beta < 1$ (the \emph{subcritical}
or \emph{high temperature}  case) we have
\begin{equation}\label{eq:kahanepeyrier}
(\E Z_{\beta,n})^{-1} Z_{\beta,n} \stackrel{n \to \infty}{\longrightarrow} Y_\beta \quad \textrm{almost surely},
\end{equation}
where the limit variable $Y_\beta$ is almost surely positive. It has recently been shown by A\"{i}d\'{e}kon and Shi (\cite{aishi11}) that for $\beta = 1$ (the \emph{critical case}),
\begin{equation}\label{eq:aidekonshi}
n^\frac{1}{2} Z_{1,n} \stackrel{n \to \infty}{\longrightarrow} Y_1 \quad \textrm{in probability},
\end{equation}
where $Y_1$ is an almost surely positive random variable of infinite mean. Another recent result, due to Madaule (\cite{ma11}), shows that for $\beta > 1$ (the \emph{supercritical } or \emph{low temperature} case) we have
\begin{equation}\label{eq:madaule}
n^\frac{3\beta}{2} Z_{\beta,n} \stackrel{n \to \infty}{\longrightarrow} Y_\beta \quad \textrm{in distribution}
\end{equation}
for a positive random variable $Y_\beta$. In the case of a Gaussian $\xi$ similar results on critical and supercritical cases were  obtained independently by the fifth author (\cite{we11}) who proved convergence in distribution. It is known that convergence in \eqref{eq:aidekonshi} (resp. \eqref{eq:madaule}) cannot improved to almost sure convergence (resp. convergence in probability).

In accordance with these deterministic normalizations for $Z_{\beta,n}$ we study the measures $\mu_{\beta,n}$ on $[0,1]$ defined by
\begin{eqnarray}\label{eq:n-levelmeasures}
\mu_{\beta,n}(I_\sigma) &=&  \left( \E Z_{\beta,n} \right)^{-1} e^{\beta X_\sigma} \quad \textrm{for }\; \beta <1,\nonumber\\
\mu_{1,n}(I_\sigma) & =& n^\frac{1}{2} e^{X_\sigma}, \quad \textrm{and}\nonumber\\ \quad \mu_{\beta,n}(I_\sigma) &=&  n^\frac{3\beta}{2} e^{\beta X_\sigma}  \quad \textrm{for }\; \beta >1,\nonumber
\end{eqnarray}
where $I_\sigma$ is the dyadic interval naturally coded by $\sigma  \in \Sigma_n$, and the density of $\mu_{\beta,n}$ with respect to the Lebesgue measure is constant on each of these level $n$ intervals. As said in the introduction, the corresponding limit measures $\mu_\beta$ in the subcritical case have been much studied and well understood, and it holds that
\begin{equation}\label{eq:subcritical}
\mu_{\beta,n} \stackrel{w}{\underset{n \to \infty}{\longrightarrow}}\mu_\beta \quad \textrm{almost surely} \quad {\rm for } \;\; \beta<1,
\end{equation}
where the law of the limit measure satisfies
\begin{equation}\label{eq:subcscaling}
\left( \mu_\beta(I_\sigma) \right)_{\sigma \in \Sigma_n} \eqlaw \left( (\E Z_{\beta,n})^{-1} e^{\beta X_\sigma} Y_\beta^{(\sigma)} \right)_{\sigma \in \Sigma_n} \quad \textrm{for all}\;\, n\geq 1,
\end{equation}
where $\{Y_\beta^{(\sigma)}\}_{\sigma \in \Sigma_n}$ is an independent collection of copies of $Y_\beta$ that is also independent of $\{X_\sigma\}_{\sigma \in \Sigma_n}$. If for $s\in \mathbb{R}$ we set 
\begin{equation}\label{phi}
\phi (s)=-\log_2\E (e^{s\xi}) \quad \text{and} \quad \widetilde \phi(s)=1-\phi(s)
\end{equation}
we see that $\E Z_{n,\beta} = 2^{n \widetilde \phi(\beta)}$.

 In the critical and supercritical cases
it was observed by Barral, Rhodes and Vargas in \cite{barhova12} that the deterministic normalizations for the partition functions 
imply the weak convergence of the measures $\mu_{\beta,n}$ to nontrivial limit measures:
\begin{eqnarray}
&&\mu_{1,n} \stackrel{w}{\underset{n \to \infty}{\longrightarrow}} \mu_1 \quad \textrm{in probability} \quad {\rm and}\label{eq:critical} \\
&&\mu_{\beta,n} \stackrel{w}{\underset{n \to \infty}{\longrightarrow}} \mu_\beta \quad \textrm{in distribution} \quad\text{for} \;\; \beta >1,\label{eq:supercritical}
\end{eqnarray}
where the laws of the measures $\mu_\beta$ for $\beta\geq 1$ can be described by
\begin{equation}\label{eq:semi-stable}
\left( \mu_\beta(I_\sigma) \right)_{\sigma \in \Sigma_n} \eqlaw \left( e^{\beta X_\sigma} Y_\beta^{(\sigma)} \right)_{\sigma \in \Sigma_n},\quad \textrm{for all}\;\, n\geq 1,
\end{equation}
where $\{Y_\beta^{(\sigma)}\}_{\sigma \in \Sigma_n}$ is an independent collection of copies of $Y_\beta$ that is also independent of $\{X_\sigma\}_{\sigma \in \Sigma_n}$. 

Barral, Rhodes and Vargas noted furthermore that the law of the supercritical limit measures $\mu_\beta$, $\beta > 1$, may also be described in terms of the law of the critical case $\mu_1$ as follows. For $\alpha \in (0,1)$, let $(L_\alpha(s))_{s \geq 0}$ be a stable L\'evy subordinator
of index $\alpha$, that is, a process with $L_\alpha(0)=0$ that has independent and
stationary increments that are characterized by the Laplace transform 
$$
\E e^{-uL_\alpha(s)}=e^{-su^{\alpha}}.
$$
Then, if $\E e^{(\beta+\epsilon)X}<\infty$ for some $\epsilon>0$, 
\begin{equation}\label{lowmeasure}
\Big(\mu_\beta([0,t])\Big)_{t \in [0,1]} \eqlaw \left(cL_{\frac{1}{\beta}}(\mu_1([0,t]))\right)_{t \in [0,1]}
\end{equation}
where $L_{\frac{1}{\beta}}$  is taken independent of the critical case measure $\mu_1$, and $c>0$ is a constant that depends only on $\beta$. Since the process $(L_\frac{1}{\beta}(s))_{s \geq 0}$ is a pure jump process, this implies that the measures $\mu_\beta$, $\beta > 1$, are almost surely purely atomic. For this reason the
phase transition in the cascade model at $\beta=1$ is called in the physics literature
{\it freezing} transition.

It thus remains to understand the critical measure $\mu_1$ which is the topic of the present paper.
Our first result shows that $\mu_1$ has no atoms:\begin{theorem}\label{th:elementary} \quad For any $\gamma\in [0,1/2)$ we have
\begin{equation}\label{eq:raja1}
n^\gamma \max_{\sigma\in \{ 0,1\}^n}\mu_1(I_\sigma)\stackrel{\Prob}{\longrightarrow} 0 \quad {\rm as} \quad n \to \infty,
\end{equation}
and for any $\gamma \in (1/2,\infty)$ we have
\begin{equation}\label{eq:raja2}
n^\gamma \max_{\sigma\in \{0,1\}^n} \mu_1(I_\sigma) \stackrel{\Prob}{\longrightarrow} \infty \quad {\rm as} \quad n \to \infty.
\end{equation}
\end{theorem}

The proof of this result, given in Section \ref{se:eleproof},  uses only elementary tools in combination with two ingredients: the exact asymptotics of the tail of the random variable $Y_1$ (Theorem \ref{th:tail} below, due to Buraczewski in \cite{bu09}), and the knowledge of the normalization considered above that is included in the very definition of $Y_{1}$ (Theorem \ref{th:deterministic} below).

\begin{corollary}\label{co:noatoms}
Almost surely the limit measure $\mu_{1}$ has  no atoms.
\end{corollary}

The proof of Theorem \ref{th:elementary}, in combination with suitable   moment estimates (see (\ref{eq2}) below)  in fact gives us a stronger result, an estimate for the modulus of continuity of the cumulative distribution function of $\mu_1$.

\begin{theorem}\label{th:modulus} Assume that $\xi$ satisfies in addition to
\eqref{eq:momentconditions} also $ \E e^{-h\xi}<\infty$ for some $h>0$. Then for any $\gamma\in (0,1/2)$
\begin{equation}\label{eq:modulus}
\mu_1(I)\leq C(\omega)\left(\log \left(1+\frac{1}{|I|}\right)\right)^{-\gamma}
\end{equation}
for all subintervals $I\subset [0,1].$ Here $C(\omega)$ is a random constant, finite almost surely. Moreover, one cannot take $\gamma >1/2$ in the above statement.
\end{theorem}

In the subcritical case $\beta\in(0,1)$, it is known that  the measure $\mu_\beta$ has a H\"older modulus of continuity. Indeed, the multifractal formalism (see Section \ref{se:multifractal} and e.g. \cite{Ba00}) tells you via Legendre transform that uniform H\"older continuity holds with any exponent $\gamma <\widetilde \phi(\beta)$ and cannot hold for any  $\gamma>\widetilde \phi(\beta)$. It turns out that our proof of Theorem \ref{th:modulus} can also be applied to considerably sharpen this result, which up to now has been the best modulus of continuity estimate in the subcritical case.

\begin{theorem}\label{th:submodulus}
Assume  only that $\E |\xi|^3 e^\xi < \infty$. 
Let $\beta \in (0,1)$ and $\gamma\in (0,1/2)$. Suppose that there exists $q_\beta>1$ such that $\widetilde\phi(\beta q_\beta)-q_\beta\widetilde\phi(\beta)=0$ (in that case $q_\beta$ is unique and the condition amounts to saying that for $q>0$ one has  $\mathbb {E}Y_\beta^q<\infty$ if and only if $q<q_\beta$). Suppose also that $\mathbb{E} \max(0,\xi) e^{\beta q_\beta\xi}<\infty$. Then for 
all subintervals $I\subset [0,1]$
\begin{equation}\label{eq:submodulus}
\mu_\beta(I)\leq C(\omega)|I|^{\widetilde\phi(\beta)}\left(\log \left(1+\frac{1}{|I|}\right)\right)^{-\gamma\beta}
\end{equation}
where $C(\omega) < \infty$  almost surely. In the Gaussian case,  $\widetilde\phi(s)=(1-\beta)^2$ and $q_\beta=1/\beta^2$. 
\end{theorem}

In the Gaussian case there is another, more sophisticated way to get hold on the size of 
$\max_{\sigma\in \{ 0,1\}^n}\mu_n(I_\sigma)$ that is based on the following result.
\begin{theorem}\label{th:exact}
 Let $\xi$ be Gaussian and assume that $\beta> 1.$ Then there is a deterministic  bounded sequence $c(n)$, bounded away from $0$ such that
\begin{equation}
c(n)\sum_{\sigma\in\Sigma_n}\left(n^{1/2 }e^{X_\sigma}Y_1^{(\sigma)}\right)^{\beta}\stackrel{d}\to  Y_\beta.
\end{equation}
\end{theorem}
The proof of this result is more technical (see Section \ref{se:generating} below) and applies the generating function techniques used in \cite{we11}. With more work one could remove the bounded sequence $c(n)$, but this formulation is sufficient for the analysis of the local behavior of $\mu_1$ carried out in Section \ref{se:localbehavior}.

Next we turn to the KPZ formula relating  the Hausdorff dimension of fractals in $[0,1]$  in the Euclidean metric and their dimension under a random metric. Specifically, for $\beta\in (0,1]$ define  on $[0,1]$  the random metric $\rho_\beta(x,y)=\mu_\beta([y,x])$ for $0\le x\le y\le 1$.  

For KPZ relations we replace \eqref{eq:momentconditions} by the weaker condition
\begin{equation}\label{phi2}
 \E \, \xi^2 e^\xi<\infty.
\end{equation}
The following result is an extension of the result   by Benjamini and Schramm \cite{BS09} on $\beta<1$ to the critical case $\beta=1$.
\begin{theorem}
\label{th:BS}
Suppose that $\phi(-s)>-\infty$ for all $s\in (0,1/2)$. Let $K\subset[0,1]$ be some (deterministic) nonempty Borel set, let $\zeta_0$ denote its Hausdorff dimension with respect to the Euclidean metric, and let $\zeta$ denote its Hausdorff dimension  with respect to the random metric $\rho_1$. Then a.s. $\zeta$  is the unique solution of the equation
\begin{equation}\label{KPZ}
\zeta_0=\phi(\zeta)
\end{equation}
in $[0,1]$.
\end{theorem}
In Section \ref{se:kpz} we also extend the KPZ relation to the supercritical case $\beta>1$. In that case, since the measure $\mu_\beta$ is discrete, $\rho_\beta$ is not a metric anymore; nevertheless $\mu_\beta$ can be formally used in the same way as if $\rho_\beta$ were a metric to define the   Hausdorff dimension  of sets $K$ which do not contain any atom of $\mu_\beta$ almost surely. In that case the Hausdorff dimension of $K$ relative to $\mu_\beta$ is the unique solution of $\zeta_0=\phi(\beta \zeta)$ almost surely, Theorem \ref{KPZalpha}.

\medskip

We refer to  Section \ref{se:multifractal} for precise statements on the multifractal spectra of the measures $\mu_\beta$ in case $\beta\geq 1.$ Finally, in Section \ref{se:localbehavior} we consider the almost sure $\mu_1$-almost everywhere local behavior of $\mu_1$ in the case of a Gaussian $\xi$.

\medskip

We close this section by comparing the phase transition of the 
 cascade measures to that of  the Gibbs measures of 
 Random Energy Models (REM's). In Derrida's REM \cite{Derrida} the $X_\sigma$, $\sigma\in\Sigma_N$ are taken
 i.i.d. Gaussian with variance $N$ (in our normalization where $\beta_c=1$) and  $\mu_{\beta,N}$ is normalized 
 to be a  probability measure  i.e. one considers  the Gibbs measure  $\tilde\mu_{\beta,N}(I_\sigma)=\tilde{Z}_{\beta,N}^{-1}e^{\beta X_\sigma}$
 with $\tilde Z_{\beta,N}=\sum_\sigma e^{\beta X_\sigma}$. 
REM is a simple model of a disordered spin system where $X_\sigma$
 is the (random) energy of the spin configuration $\sigma\in\{0,1\}^N$. In REM  the energies
are independent whereas in  the cascade model they are strongly correlated.
  
 REM  has a freezing transition very similar to the cascade model. In \cite{bov} it is proven
 that for $\beta\leq\beta_c$, $\tilde{\mu}_{\beta,N}\to \tilde{\mu}_{\beta}$ almost surely  as $N\to\infty$ and the
 limit measure $\tilde{\mu}_{\beta}$ is the  Lebesgue measure. For 
$\beta>\beta_c$  \cite{bov} prove the analogue of \eqref{lowmeasure} namely that $\tilde{\mu}_{\beta,N}$
converges in distribution to $\tilde{\mu}_{\beta}$ given by
$
\tilde{\mu}_{\beta}([0,t])\stackrel{d}= {L_{\frac{1}{\beta}}(t)}/{L_{\frac{1}{\beta}}(1)}$. This result is the REM analogue of 
\eqref{lowmeasure} since $\tilde{\mu}_1[0,t]=t$. %In the cascade case the critical measure ${\mu}_1$ is a random measure.
  It is actually known (see \cite{bk}, \cite{ABK}) that the Gibbs
weights in REM and cascade when ordered in decreasing size converge to the same Poisson-Dirichlet
process. However, the result  \eqref{lowmeasure} is more general as it also gives locations of the atoms.

\section{Proof of Theorem \ref{th:elementary}}\label{se:eleproof}

Before the proof of Theorem \ref{th:elementary} we restate the result on the deterministic normalization that is needed to make the critical case partition function $Z_{1,n}$ converge to a nontrivial random variable in the $n \to \infty$ limit.

\begin{theorem}[A\"{i}d\'{e}kon and Shi \cite{aishi11}]\label{th:deterministic}
Assume  that $\E\, \xi^2 e^\xi <
\infty$. Then
$$
n^{1/2} Z_{1,n} =n^{1/2}\sum_{\sigma\in\Sigma_n}e^{X_\sigma}\stackrel{\Prob}{\rightarrow} Y_1 \quad {\rm as}\quad n\to\infty,
$$
where almost surely $0<Y_1<\infty$.
\end{theorem}
We also need the following result of Buraczewski \cite{bu09} that generalizes Guivarch's tail asymptotics of subcritical Mandelbrot cascades \cite{gu90} to the critical case.
\begin{theorem}[Buraczewski \cite{bu09}]\label{th:tail} Assume \eqref{eq:momentconditions} and that the distribution of $\xi$ is nonlattice.
Then the distribution function $h(x):=\Prob (Y_1>x)$ satisfies
$$
h(x) \sim \frac{d}{x} \quad {\rm as} \quad x \to \infty,
$$
i.e. there is a positive constant $d$ such that $\lim_{x\to\infty}xh(x)=d$.

\noindent If the distribution of $\xi$ is lattice (i.e. supported on some arithmetic sequence), one has
$0<d_1\leq xh(x)\leq d_2<\infty$ for large $x$.
\end{theorem}

\begin{proof}[Proof of Theorem {\rm \ref{th:elementary}}.] \quad
We prove \eqref{eq:raja1} first. Fix $\gamma\in (0,1/2)$ and let 
$\delta_1>0$. By Theorem \ref{th:tail} there exists $x_1>0$ such that $h(x)\le (d+\delta_1)/x\le 1/2$ for all $x>x_1$, and Theorem \ref{th:deterministic} shows that
\begin{equation}\label{eq:max1}
\Prob (\Bcal_n)\to 1\quad \textrm{as} \quad n\to\infty, \quad \textrm{where} \quad \Bcal_n:= 
 \big\{ \max_{\sigma \in \Sigma_n} e^{X_\sigma} \leq n^{-\gamma}/x_1 \big\}.
\end{equation}
Recall that by \eqref{eq:semi-stable} the law of the measure $\mu_1$ satisfies
%By noting that  $\lim_{m\to\infty} \big(\frac{n+m}{m}\big)^{1/2}=1$ and using the very definition of the measure $\mu_1$  we see that
$$
\left(\mu_1(I_\sigma)\right)_{\sigma\in \Sigma_n} \eqlaw \left(e^{X_\sigma}Y_1^{(\sigma)}\right)_{\sigma\in \Sigma_n},
$$
where $\{Y_1^{(\sigma)}\}_{\sigma\in\Sigma_n}$ is a family of independent copies of $Y_1$, also independent from $\{X_\sigma\}_{\sigma\in\Sigma_n}$.
Using  this independence we may compute
\begin{eqnarray*}
\Prob\left( \max_{\sigma \in \Sigma_n} e^{X_\sigma} Y_1^{(\sigma)} < n^{-\gamma} \, \Big\vert \, \{X_\sigma\} \right) & = & \prod_{\sigma \in \Sigma_n} \left( 1 - \Prob\left( Y_1^{(\sigma)} \geq n^{-\gamma} e^{-X_\sigma} \, \Big\vert \, \{X_\sigma\} \right) \right) \\
& = & \prod_{\sigma \in \Sigma_n} \left( 1 - h\left(n^{-\gamma} e^{-X_\sigma}\right) \right).
\end{eqnarray*}
By taking expectations this gives
\begin{equation}\label{eq:atomprob}
\Prob\left( \max_{\sigma \in \Sigma_n} e^{X_\sigma} Y_1^{(\sigma)} < n^{-\gamma}\right) = \E \prod_{\sigma \in \Sigma_n} \left( 1 - h\left(n^{-\gamma} e^{-X_\sigma}\right) \right).
\end{equation}
We employ the elementary inequality $1-x \geq e^{-2x}$ for $x \in [0,1/2]$ and Theorem \ref{th:tail} to obtain
\begin{equation}\label{eq:eq}
\begin{split}
\Prob\left( \max_{\sigma \in \Sigma_n} e^{X_\sigma} Y_1^{(\sigma)} < n^{-\gamma} \right)  \geq &\; \E \,  \ind_{\Bcal_n} \prod_{\sigma \in \Sigma_n} \left( 1 - h\left(n^{-\gamma} e^{-X_\sigma}\right) \right) \\
 \geq &\; \E \,  \ind_{\Bcal_n} \exp\left( - 2(d+\delta_1)n^\gamma \sum_{\sigma \in \Sigma_n} e^{X_\sigma} \right).
\end{split}
\end{equation}
By Theorem  \ref{th:deterministic} we have $n^\gamma  \sum_{\sigma \in \Sigma_n} e^{X_\sigma}\to 0$ in   probability. By using this fact together with (\ref{eq:max1}) and the bounded convergence theorem we deduce that
$$
\Prob\left( \max_{\sigma \in \Sigma_n} e^{X_\sigma} Y_1^{(\sigma)} < n^{-\gamma} \right) \longrightarrow 1 \quad \textrm{as } n \to\infty.
$$
This proves \eqref{eq:raja1}.

To prove \eqref{eq:raja2} we let $\gamma > 1/2$ and use the estimate $1-x \leq e^{-x}$ for $x \geq 0$ in \eqref{eq:atomprob} to get
$$
\Prob\left( \max_{\sigma \in \Sigma_n} e^{X_\sigma} Y_1^{(\sigma)} < n^{-\gamma}\right) \leq \E \exp\left( -\sum_{\sigma \in \Sigma_n} h\left(n^{-\gamma} e^{-X_\sigma}\right)\right).
$$
We deduce from \eqref{eq:madaule}, or alternatively from \cite{we11} in the Gaussian case, that    $n^{\frac{3}{2}-\varepsilon}\max_{\sigma \in \Sigma_n} e^{X_\sigma}\underset{n\to\infty}{\stackrel{d}{\longrightarrow}} 0$ for any $\varepsilon > 0$, whence 
$$
\Prob\left( \max_{\sigma \in \Sigma_n} e^{X_\sigma} \leq n^{-\frac{3}{2}+\varepsilon}\right) \stackrel{n\to\infty}{\longrightarrow} 1.
$$
Especially it follows that if $\gamma\in (1/2,1)$ we have
$$
\Prob(\Acal_n) \stackrel{n \to \infty}{\longrightarrow} 1 \quad {\rm for} \quad \Acal_n := \left\{ n^{1/2} < n^{-\gamma} \min_{\sigma \in \Sigma_n} e^{-X_\sigma} \right\}.
$$
By Theorem \ref{th:tail}, we may find $\delta_2 \in (0,d)$ and $x_2>0$ such that $h(x)> (d-\delta_2)/x$ for all $x>x_2$. Consequently,   we have 
$$
\Prob\left( \max_{\sigma \in \Sigma_n} e^{X_\sigma} Y_1^{(\sigma)} < n^{-\gamma}\right) \leq \E \ind_{\Acal_n} \exp\left( - (d-\delta_2) n^\gamma \sum_{\sigma \in \Sigma_n} e^{X_\sigma} \right) + (1-\Prob(\Acal_n))
$$
for all large $n$. Since $\gamma > 1/2$, Theorem \ref{th:deterministic}, the bounded convergence theorem, and the fact that $\Prob(\Acal_n) \to 1$ imply that
$$
\Prob\left( \max_{\sigma \in \Sigma_n} e^{X_\sigma} Y_1^{(\sigma)} < n^{-\gamma}\right) \longrightarrow 0 \quad {\rm as} \quad n \to \infty.
$$
This concludes the proof of \eqref{eq:raja2} for $\gamma\in (1/2,1)$, and after this the case $\gamma\ge 1$ is obvious.

\end{proof}

\section{Proofs of Theorems \ref{th:modulus} and \ref{th:submodulus}.}\label{se:modulus}

In this section we employ the notations used in the proof of Theorem \ref{th:elementary}. In addition, we denote 
\begin{equation}\label{sn}
S_{n,1}:=n^{1/2}\sum_{\sigma\in\Sigma_n} e^{X_\sigma} \quad \textrm{and} \quad S_{n,\theta}:=n^{\frac{3\theta}{2}}\sum_{\sigma\in\Sigma_n} e^{\theta X_\sigma} \textrm{ for } \theta > 1.
\end{equation}
For all $\theta \geq 1$ and $\varepsilon > 0$ we shall need the estimates
\begin{equation}\label{eq1}
\Prob (S_{n,\theta}>\lambda)\leq C_{\theta,\varepsilon} \lambda^{-(1-\varepsilon)/\theta} \quad {\rm with}\;\; C_{\theta,\varepsilon}\;\; \textrm{independent of}\;\; n\geq 1.
\end{equation}
In the case $\theta > 1$ this estimate is stated explicitly as Proposition 2.1 in \cite{ma11} for general $\xi$ satisfying $\E \xi^3 e^\xi < \infty$. For $\theta = 1$ we make the stronger assumptions \eqref{eq:momentconditions} and $\E e^{-h \xi} < \infty$ for some $h > 0$. In this case, Theorem 1.5 of Hu and Shi (\cite{hushi09}) states that for any $\varepsilon \in (0,1)$ there exists a constant $C_{1,\varepsilon} > 0$ such that
\begin{equation}\label{eq2}
\E (S_{n,1})^{1-\varepsilon}\leq C_{1,\varepsilon}\quad {\rm for\; all}\;\; n\geq 1.
\end{equation}
By Chebyshev's inequality the tail estimate \eqref{eq1} follows immediately.

In the case of a Gaussian $\xi$ we give, in Section \ref{se:generating}, another approach to establishing the estimate \eqref{eq1}, perhaps simpler than the one used in \cite{ma11} and \cite{hushi09}. By using the generating function techniques of \cite{we11} we will study the modulus of continuity of the Laplace transform of $S_{n,\theta}$ at $0$ and prove the following lemma.
\begin{lemma}
\label{le:phimodulus}
Suppose $\xi$ is Gaussian. Let $\theta \geq 1$ and denote the Laplace transform of $S_{n,\theta}$ by
\begin{equation}\label{phin}
\phi_{n,\theta}(t):=\E \exp (-tS_{n,\theta}). 
\end{equation}
Then for all  $\varepsilon >0$ there is a constant $C_{\theta,\varepsilon}>0$, independent of $n$, such that 
\begin{equation}\label{eq3}
1-\phi_{n,\theta}(t)\leq C_{\theta,\varepsilon} t^{(1-\varepsilon)/\theta}\quad {\rm for\; all}\;\;t>0.
\end{equation}
\end{lemma}
\noindent The uniform estimate (\ref{eq1}) then follows from Chebyshev's inequality and the formula
$$
\E (S_{n,\theta})^{(1-\varepsilon)/\theta}=c_{\theta,\varepsilon} \int_0^\infty t^{-1-\frac{1-\varepsilon}{\theta}}\left(1-\phi_{n,\theta}(t)\right) \d t,
$$
where the explicit expression for the constant is $c_{\theta,\varepsilon} = \frac{\theta}{1-\varepsilon} \Gamma\left(1-\frac{1-\varepsilon}{\theta}\right)$.

The following lemma contains the essential probability estimate needed for the theorems.
\begin{lemma}\label{lemma:modulus-probability-estimate}
Let $\beta \in (0,1]$. In the case $\beta=1$ assume that $\xi$ satisfies \eqref{eq:momentconditions} and $\E e^{-h\xi} < \infty$ for some $h > 0$, and in the case $\beta \in (0,1)$ assume the same properties as in Theorem~\ref{th:submodulus}. Denote $\alpha_1 = 1/2$ and $\alpha_\beta = 3/2$ for $\beta < 1$, and let $\gamma \in (0,\alpha_\beta)$. For every $\varepsilon > 0$ there exists a constant $C_\varepsilon > 0$ depending only on $\beta$ and $\varepsilon$ such that
$$
\Prob\left( \max_{\sigma \in \Sigma_n} \mu_\beta(I_\sigma) \geq 2^{-n\widetilde\phi(\beta)} n^{-\gamma \beta} \right) \leq C_\varepsilon n^{(1-\varepsilon)(\gamma-\alpha_\beta)}.
$$
Especially, $2^{n\widetilde\phi(\beta)} n^{\gamma \beta} \max_{\sigma \in \Sigma_n} \mu_\beta(I_\sigma) \longrightarrow 0$ in probability
as ${n\to\infty}$.
\end{lemma}

\begin{proof}
Fix $\beta \in (0,1]$, $\alpha_\beta \in \{1/2,3/2\}$, $\gamma \in (0,\alpha_\beta)$ and $\varepsilon > 0$ as in the statement of the lemma. Recall from \eqref{eq:subcscaling} that the distribution of the measure $\mu_\beta$ is given by
$$
\left( \mu_\beta(I_\sigma) \right)_{\sigma \in \Sigma_n} \eqlaw \left( 2^{-n\widetilde\phi(\beta)} e^{\beta X_\sigma} Y_\beta^{(\sigma)} \right)_{\sigma \in \Sigma_n},
$$
where $\{Y_\beta^{(\sigma)}\}$ is an independent collection of copies of the total mass variable $Y_\beta$. Denote $h_\beta(x) = \Prob(Y_\beta \geq x)$. In the case $\beta = 1$ it follows from the theorem of Buraczewski (Theorem \ref{th:tail}) and in the case $\beta < 1$ from our assumptions and the work of Guivarch \cite{gu90} that there exist constants $x_\beta,d_\beta > 0$ such that
\begin{equation}\label{eq:beta-tail-estimates}
h_\beta(x) \leq \frac{d_\beta}{x^{q_\beta}} \leq \frac{1}{2} \quad \textrm{for all } x > x_\beta,
\end{equation}
where $q_1=1$ and $q_\beta$ is defined as in Theorem~\ref{th:submodulus} ($q_\beta$ is necessarily unique due to the strict convexity of $\widetilde \phi$ and the fact that $q=1$ is another solution of $\widetilde\phi (\beta q)-q\widetilde\phi(\beta)=0$). 
Define the events
$$
\Bcal_{n,\beta} = \left\{ \min_{\sigma \in \Sigma_n} n^{-\gamma \beta} e^{-\beta X_\sigma} > x_\beta \right\} = \left\{ \max_{\sigma \in \Sigma_n} e^{\beta X_\sigma} < n^{-\gamma \beta}/x_\beta \right\}.
$$
Since
$$
\Bcal_{n,\beta} \supset \left\{ \sum_{\sigma \in \Sigma_n} e^{\beta  q_\beta X_\sigma} < n^{-\gamma \beta q_\beta}/x_\beta^{q_\beta} \right\},
$$
by the estimate \eqref{eq1} we have (noticing that $\beta q_\beta>1$ since $0\le \widetilde\phi(1)\le \widetilde\phi (q')<\widetilde\phi(\beta )$ for all $q'\in (\beta,1]$)
\begin{equation}\label{eq:bcal-prob-estimate}
1 - \Prob(\Bcal_{n,\beta}) \leq \Prob\left( S_{n,\beta q_\beta} \geq n^{(\alpha_\beta-\gamma)\beta q_\beta}/x_\beta^{q_\beta} \right) \leq C n^{(1-\varepsilon)(\gamma-\alpha_\beta)}
\end{equation}
for some constant $C > 0$ depending on $\beta$ and $\varepsilon$.

Using the estimates \eqref{eq:beta-tail-estimates} we may perform a computation similar to the one used to obtain \eqref{eq:eq} in the proof of Theorem \ref{th:elementary}. The resulting estimate~is
\begin{eqnarray*}
\Prob\left( \max_{\sigma \in \Sigma_n} \mu_\beta(I_\sigma) < 2^{-n\widetilde\phi(\beta)} n^{-\gamma \beta} \right) %& = & \Prob\left( \max_{\sigma \in \Sigma_n} 2^{-n(1-\beta)^2} e^{\beta X_\sigma} Y_\beta^{(\sigma)} < 2^{-n(1-\beta^2)} n^{-\gamma \beta} \right) \\
& = & \Prob\left( \max_{\sigma \in \Sigma_n} e^{\beta X_\sigma} Y_\beta^{(\sigma)} < n^{-\gamma \beta} \right) \\
& = & \E \prod_{\sigma \in \Sigma_n} \left( 1 - h_\beta \left( n^{-\gamma \beta} e^{-\beta X_\sigma} \right) \right) \\
& \geq & \E \ind_{\Bcal_{n,\beta}} \exp\left( -2 d_\beta \sum_{\sigma \in \Sigma_n} n^{\gamma \beta q_\beta} e^{\beta q_\beta X_\sigma} \right),
\end{eqnarray*}
which in combination with \eqref{eq:bcal-prob-estimate} and \eqref{eq1} yields
\begin{eqnarray*}
\lefteqn{\Prob\left( \max_{\sigma \in \Sigma_n} \mu_\beta(I_\sigma) \geq 2^{-n\widetilde\phi(\beta)} n^{-\gamma \beta} \right) } \hspace{3.0cm} \\
& \leq & 1 - \Prob(\Bcal_{n,\beta}) + 1 - \E \exp\left( -2 d_\beta \sum_{\sigma \in \Sigma_n} n^{\gamma \beta q_\beta} e^{\beta q_\beta X_\sigma} \right) \\
& = & 1 - \Prob(\Bcal_{n,\beta}) + 1 - \E \exp\left( -2 d_\beta n^{(\gamma-\alpha_\beta)\beta q_\beta} S_{n,\beta q_\beta} \right) \\
& \leq & C' n^{(1-\varepsilon)(\gamma-\alpha_\beta)},
\end{eqnarray*}
where the constant $C' > 0$ depends only on $\beta$ and $\varepsilon$. The proof of the lemma is complete.
\end{proof}

Theorems \ref{th:modulus} and \ref{th:submodulus} now follow easily.

\begin{proof}[Proof of Theorem {\rm \ref{th:submodulus}}.]
Let $\beta \in (0,1)$ and $\gamma \in (0,1/2)$. As it is clearly enough to consider dyadic intervals in \eqref{eq:submodulus}, all we need to do is to improve the convergence in probability in Lemma \ref{lemma:modulus-probability-estimate} to almost sure convergence. But since $\gamma-3/2 < -1$, we may take an $\varepsilon > 0$ small enough that
$$
\sum_{n=1}^\infty \Prob\left( \max_{\sigma \in \Sigma_n} \mu_\beta(I_\sigma) \geq 2^{-n\widetilde\phi(\beta)} n^{-\gamma \beta} \right) \leq C_\varepsilon \sum_{n=1}^\infty n^{(1-\varepsilon)(\gamma-\frac{3}{2})} < \infty.
$$
The Borel--Cantelli lemma now implies the existence of a random constant $C(\omega) < \infty$ such that
$$
\max_{\sigma \in \Sigma_n} \mu_\beta(I_\sigma) < C(\omega) 2^{-n\widetilde\phi(\beta)} n^{-\gamma \beta} \quad \textrm{for all }n=1,2,\dots,
$$
which is the desired conclusion.
\end{proof}

\begin{proof}[Proof of Theorem {\rm \ref{th:modulus}}.]
Let $\gamma \in (0,1/2)$ and choose an integer $\ell \geq 1$ so that $\ell(\gamma-1/2) < -2$. Applying Lemma \ref{lemma:modulus-probability-estimate} with $\varepsilon = 1/2$ gives
$$
\sum_{k=1}^\infty \Prob\left( \max_{\sigma \in \Sigma_{k^\ell}} \mu_1(I_\sigma) \geq (k^\ell)^{-\gamma} \right) \leq C_{1/2} \sum_{k=1}^\infty k^{\ell \frac{\gamma-1/2}{2}} < \infty.
$$
Borel--Cantelli lemma then implies
$$
\max_{\sigma \in \Sigma_{k^\ell}} \mu_1(I_\sigma) \leq C'(\omega) (k^\ell)^{-\gamma} \quad \textrm{for a random constant } C'(\omega) < \infty
$$
for the dyadic intervals of levels $1^\ell, 2^\ell, 3^\ell,\dots$. It remains to note that the sequence of maxima $(\max_{\sigma\in\Sigma_n} \mu_1(I_\sigma))_{n=1}^\infty$ is decreasing, so for $k^\ell \leq n < (k+1)^\ell$ we have
$$
\max_{\sigma\in\Sigma_n} \mu_1(I_\sigma) \leq \max_{\sigma\in\Sigma_{k^\ell}} \mu_1(I_\sigma) \leq C'(\omega) k^{-\ell \gamma} < C'(\omega) 2^{\ell \gamma} n^{-\gamma}.
$$
This shows that the estimate \eqref{eq:modulus} indeed holds. Finally, the statement that one
cannot take $\gamma > 1/2$ in the result is an immediate consequence of the divergence result \eqref{eq:raja2} in Theorem \ref{th:elementary}.
\end{proof}

\section{KPZ relations}\label{se:kpz}
\subsection{The KPZ formula associated with $\mu_1$}
 In the case where $\xi$ is Gaussian the relation  stated by the Theorem \ref{th:BS} is precisely the KPZ formula predicted by physicists working in quantum gravity. Benjamini and Schramm \cite{BS09} proved Theorem \ref{th:BS} for the random metrics defined by the one-dimensional Mandelbrot cascade measures,
 i.e. in the subcritical case
 \begin{equation}\label{eq:subcriticalnormalization}
\E e^\xi = \frac{1}{2} \quad {\rm and} \quad \E \, \xi e^\xi < 0,
\end{equation} 
which basically corresponds to the measures
$\mu_\beta$ for $\beta <1$ considered in Section \ref{se:defres}. They were inspired by a different point of view developed by Duplantier and Sheffield \cite{DS08}, who  gave a sense to the KPZ formula in terms of expected box counting dimension in the context of Liouville quantum gravity,  by considering random measures associated with the Gaussian free field. 

In addition, Rhodes and Vargas \cite{rhovar08} derived in dimension 1 a relation similar to \eqref{KPZ}  between Hausdorff dimensions when comparing Euclidean geometry and the geometry given by the random metric associated with the limit measure of a non-degenerate infinitely divisible cascade. In higher dimension they obtained such a formula  by using the Lebesgue measure and the random measure, not as metric, but as functions of balls to define Hausdorff measures and dimension  (notice that in the log-Gaussian case, the multiplicative chaos of \cite{rhovar08} and the measures considered in \cite{DS08} are closely related).   

Before starting the proof of Theorem \ref{th:BS}, let us first comment on the difference between our assumptions on moments of negative orders of $e^\xi$ and those made in \cite{BS09} in the subcritical case. If we denote by $\widetilde{\mu}_1$ the measure $\widetilde\mu_1=2e^{\widetilde \xi}\mu_1$, where $\widetilde\xi$ is a copy of $\xi$ independent of $\mu_1$, then $\widetilde Y_1=\|\widetilde \mu_1\|_{{\rm TV}}$ satisfies the same functional equation as the total mass of the Mandelbrot measure considered in  \cite{BS09}. Then it follows from  \cite{BS09} that for $s>0$ we have $\E \widetilde Y_1^{-s} <\infty$ if and only if $\phi(-s)>-\infty$, and for the proof of KPZ formulas one needs $\E \widetilde Y_1^{-s} <\infty$ for all $s\in (0,1)$. Due to our definition of $\mu_1$,  setting $Y_1=\| \mu_1\|_{{\rm TV}}$ 
we have  \cite[Theorem 4(a)]{Molchan96} 
\begin{equation}\label{eq:negmoments}
\mathbb{E}Y_1^{-s}<\infty\ (\forall\, s\in (0,1)) \quad \textrm{as soon as }\;\phi(-s/2)>-\infty \ (\forall\, s\in (0,1/2))
\end{equation}
where $\phi (s)=-\log_2\E (e^{s\xi})$.
Moreover,
$\E (Y_1^s)<\infty$ for $s\in (0,1)$, see \cite[Theorem 4 and 5]{bigky05}.
  
\begin{proof}[Proof of Theorem \ref{th:BS}] 
After the above observations, the proof of the KPZ formula established in \cite{BS09} can be mimicked, up to changes imposed by the fact that $\E(Y_1^s)<\infty$ for $s\in (0,1)$ but not for $s=1$, and  that we make precise below. 

Employing the notation of \cite{BS09}, let us denote by $\ell$ the total mass $Y_1$ of  $\mu_1$. We use the following form of Lemma 3.3 in \cite{BS09}: Let $x,y\in [0,1]$ and let $s\in (0,1)$. Then $\E(\rho_1(x,y)^s)\le 8|x-y|^{\phi(s)}\E(\ell^s)$.

For  the upper bound for the Hausdorff dimension, which corresponds to \cite[Theorem 3.4]{BS09}, due to the above analogue of \cite[Lemma 3.3]{BS09} the proof is the same  in case  $\zeta_0<1$. In turn, the case $\zeta_0=1$ is trivial.

For the lower bound, using the same notations as in \cite{BS09}, we fix a non-empty Borel set $K$ such that $\zeta_0>0$ and $t\in (0,\zeta_0)$. We set $s=\phi^{-1}(t)$, denote by $\nu_0$ a positive Borel measure  carried by $K$ and such that $\mathcal E_t(\nu_0)=\int\int\frac{\nu_0(dx)\nu_0(dy)}{|y-x|^t}<\infty$, and consider the sequence of measures  $(\nu_n)_{n\ge 1}$ whose densities with respect to $\nu_0$ are given by $e^{s X_\sigma}/(\E e^{s\xi})^{|\sigma|}$ over each interval $I_\sigma$. Here $|\sigma|=n$ for $\sigma\in\Sigma_n$. 

Also, we consider  $\rho_{1,n}(x,y)=\max (\rho_1(x,y),\mu_1(I_n(x)),\mu_1(I_n(y)))$ for all $x,y\in [0,1]$, where $I_n(x)$ stands for the closure of  the semi-open to the left dyadic interval of generation $n$ containing $x$, and $[1-2^{-n},1]$ if $x=1$. Then, following \cite[Theorem 3.5]{BS09} proof, we get 
$$
\E\Big (\mathcal E_s(\nu_n,\rho_{1,n})=\int\int\frac{\nu_n(dx)\nu_n(dy)}{\rho_{1,n}(x,y)^s}\Big )= O(1) \mathcal E_t(\nu_0).
$$

 Then, noting that $\nu_n([0,1])^2\ell^{-s}\le \mathcal E_s(\nu_n,\rho_{1,n})$, at the end of the proof of \cite[Theorem 3.5]{BS09},  H\"older's inequality must be applied as follows, with $h\in (1+s,2)$:
\begin{align*}
&\E\Big[ \nu_n([0,1])^{h/(1+s)}\Big]= \E\Big [\Big (\nu_n([0,1])^{h/1+s}\ell^{-hs/2(1+s)}\Big )\ell^{+hs/2(1+s)}\Big ]\\
 \le\; & \E\Big [\nu_n([0,1])^{h}\ell^{-hs/2}\Big ]^{1/1+s}\E(\ell^{h/2})^{s/1+s}
\le \; \E\Big [\nu_n([0,1])^{2}\ell^{-s}\Big ]^{h/2(1+s)}\E(\ell^{h/2})^{s/1+s}\\
=\; & O(1) (\mathcal E_t(\nu_0))^{h/2(1+s)}\E(\ell^{h/2})^{s/1+s},
\end{align*} 
which is finite because $\mathcal E_t(\nu_0)<\infty$ and  $\E(\ell^{h/2})^{s/1+s}<\infty$ since $h/2 < 1$. Finally, the martingale $\nu_n([0,1])$ is bounded in $L^{h/(1+s)}$, so it is uniformly integrable since $h/(1+s)>1$, and $\nu_n$ converges weakly  almost surely to a nondegenerate measure $\nu$, necessarily supported on $K$, and such that $\int\int\frac{\nu(dx)\nu(dy)}{\rho_1(x,y)^s}<\infty$, which implies that the lower Hausdorff dimension of $\nu$ with respect to $\rho_1$ is at least $s$.  Thus, the Hausdorff dimension of $K$ is almost surely  at least $s$ for all $s<\zeta$, hence the conclusion.  
\end{proof} 

\subsection{The KPZ formula associated with $\mu_\beta$}\label{subse:KPZagain}

Assume that $\alpha\in (0,1)$ and let $L_\alpha$ be a stable subordinator of index $\alpha$ independent of the $\sigma$-algebra generated by $\{\xi_\sigma: \sigma\in\Sigma\}$. We recall that  up to a multiplicative constant  $\mu_\beta$ is the measure obtained as the derivative of the function $L_{1/\beta}\circ F_{\mu_1}$ on $[0,1]$
where $F_{\mu_1}(x)=\mu_1([0,x])$.
It is also of interest to consider  the measures obtained in the same way from a subcritical cascade measure $\mu$. Altogether these measures unify stable L\'evy subordinators and Mandelbrot measures in a natural class of generalized semi-stable processes which satisfy scaling properties similar to \eqref{eq:semi-stable}.

%The interest in these random measures comes on one hand from the modelling of multifractal signals, but even more importantly because  they unify stable L\'evy subordinators and Mandelbrot measures in the same class of generalized semi-stable processes, because they share with the measures  $\nu_{\alpha,1}\sim \mu_{1/\alpha}$ scaling properties similar to \eqref{eq:semi-stable}. 

 Let us fix the conventions used in this subsection and in Section \ref{se:multifractal} below:
\begin{itemize}

 \item $\mu$ stands for the  Mandelbrot cascade measure generated by the variable $\xi$ in the subcritical case \eqref{eq:subcriticalnormalization}, so that  $\E \xi e^\xi<0$. Then we denote by  $\nu_\alpha$   the measure obtained as the derivative of % the increasing  L\'evy process in multifractal time 
$L_\alpha\circ F_\mu$. %, with $F_\mu:x\in[0,1]\mapsto \mu([0,x])$.

  \item As before $\mu_1$ stands for a critical Mandelbrot measure, whence $\E \xi e^\xi=0$.
%see \eqref{eq:criticalnormalization}. 
We denote  by $\nu_{\alpha,1}$  the measure obtained as the derivative of  $L_{\alpha}\circ F_{\mu_1}$. %, with $F_{\mu_1}:x\in[0,1]\mapsto \mu_1([0,x])$. 
Thus, up to a positive multiplicative constant,  we have $\mu_\beta=\nu_{1/\beta,1}$ for
$\beta >1$.

\end{itemize}

A natural way to extend the usual notion of box-counting and Hausdorff dimension is to replace the metric by a continuous measure \cite[p. 141]{Billingsley}. This is what was used in \cite{DS08} and \cite{rhovar08} respectively, to get the KPZ relations in dimension $\ge 2$. Thus, if $\nu$ is a positive continuous Borel measure supported on $[0,1]$, we can define for $s\ge 0$ and any subset $E$ of $[0,1]$
$$
H^s_\nu(E)=\lim_{\delta\to 0} \inf\Big\{\sum_{i\ge 1}\nu(I_i)^s: E\subset \bigcup_{i\ge 1}I_i,\ I_i\text{ interval of length }\le \delta\Big\},
$$
and the Hausdorff dimension of $E$ relative to $\nu$ as 
\begin{equation}\label{dimnu}
\dim_\nu (E)=\sup\{s\ge 0: H^s_\nu(E)=\infty\}=\inf\{s\ge 0: H^s_\nu(E)=0\}.
\end{equation}
Note that since we are in dimension 1, this definition equals the definition of the Hausdorff dimension in the metric 
$$
\rho_\nu(x,y):=\nu([x,y]).
$$ 
If the measure $\nu$ is not continuous, it is easy to check that \eqref{dimnu} is still defined if $\nu (E)=0$. Thus we can seek for analogues of the KPZ formula invoking Hausdorff dimensions relative to $\nu_\alpha$ or $\nu_{\alpha,1}$. This will use the following lemma.

\begin{lemma} \label{nullmeas}
\begin{enumerate}
\item If $E$ is a Borel set of null Lebesgue measure, then $\mu (E)=0$ and $\nu_\alpha(E)=0$ a.s.
\item Suppose that $\E \, \xi^2 e^\xi<\infty$. If $E$ is a Borel set of Euclidean Hausdorff dimension less than 1, then $\mu_1(E)=0$ and $\nu_{\alpha,1}(E)=0$ a.s. 
\end{enumerate}
\end{lemma}
\begin{proof} We only prove (2) since (1) is similar and slightly simpler.   Let $1>t>\dim K$, $\epsilon>0$, and  consider a covering of $E$ by a collection $(I_i)_{i\ge 1}$ of dyadic subintervals of $[0,1]$ such that $\sum_{i\ge 1}|I_i|^t\le \epsilon$.

The random variable $\sum_{i\ge 1} \nu_{\alpha,1}(I_i)$ is equal in distribution to $\sum_{i\ge 1}\mu_{1}(I_i)^{1/\alpha} Z_i$, where the random variables $Z_i$ are identically distributed with a  positive $\alpha$-stable random variable $Z$, and independent of $\mu_1$. It follows that 
\begin{align*}
\E\Big[\Big ( \sum_{i\ge 1} \nu_{\alpha,1}(I_i)\Big )^{t\alpha}\Big )\Big ]&\le \sum_{i\ge 1}\mathbb{E}(\mu_{1}(I_i)^{t}) \E (Z^{t\alpha})\\
&=\sum_{i\ge 1}(\E e^{t\xi})^{-\log_2|I_i|}\E (Y_1^t)\E (Z^{t\alpha})\\
&\le \sum_{i\ge 1}(\E e^{\xi})^{-t\log_2|I_i|}\E (Y_1^t)\E (Z^{t\alpha})\\
&=\E (Y_1^t)\E (Z^{t\alpha})\sum_{i\ge 1}|I_i|^t\le \E (Y_1^t)\E (Z^{t\alpha})\epsilon.
\end{align*}
Taking $\epsilon=2^{-n}$ we get a deterministic sequence of coverings $(\bigcup_{i\ge 1}I_i^{n})_{n\ge 1}$ of $E$ such that a.s. $\sum_{n\ge 1}\Big (\sum_{i\ge 1} \nu_{\alpha,1}(I^n_i)\Big )^{t\alpha}<\infty$, hence $\lim_{n\to\infty}\sum_{i\ge 1} \nu_{\alpha,1}(I^n_i)=0$. 

A similar calculation shows that $\mu_1(E)=0$ a.s. 
\end{proof}

We can now state a result regarding the KPZ relations associated with $\nu_\alpha$ or $\nu_{\alpha,1}$. 

\begin{theorem}\label{KPZalpha} Let $\alpha\in (0,1)$. Suppose that the variable $\xi$ satisfies $\phi(-s)>-\infty$ for all $s\in (0,1/2)$, where $\phi$ is defined as before by \eqref{phi}. Let $K\subset[0,1]$ be some (deterministic) nonempty Borel set and  let $\zeta_0$ denote its Hausdorff dimension with respect to the Euclidean metric.
\begin{enumerate}
\item Suppose that $K$ has  Lebesgue measure $0$. Let $\zeta$ denote the Hausdorff dimension of $K$ with respect to the random metric $\rho_\mu$ and $\zeta_\alpha$ its Hausdorff dimension relative to $\nu_\alpha$. Then a.s. $\zeta_\alpha$  is the unique solution of the equation 
$$
\zeta_0=\phi(\zeta_\alpha/\alpha)
$$
in $[0,\alpha]$, i.e. $\zeta_\alpha=\alpha \zeta$.

\medskip

\item Suppose that $\E \, \xi^2 e^\xi<\infty$ and $\zeta_0<1$.  Let $\zeta$ denote the Hausdorff dimension of $K$ with respect to the random metric $\rho_{\mu_1}$ and $\zeta_\alpha$ its Hausdorff dimension relative to ${\nu_{\alpha,1}}$. The same conclusion as in part {\rm (1)} holds.
\end{enumerate}
\end{theorem}

An analogue to Theorem~\ref{KPZalpha}.1 was first proved in \cite{bajinrhovar12} in the context of non-degenerate Kahane Gaussian multiplicative chaos, and it gave a rigorous mathematical justification to the so called dual KPZ formula. 

\begin{proof}
Note that computing $\zeta$, the Hausdorff dimension of $K$ relative to $\mu$ (resp $\mu_1$) amounts to computing the Euclidean Hausdorff dimension of the image $K_\mu$ (resp. $K_{\mu_1}$) of $K$ by $F_\mu$ (resp.  $F_{\mu_1}$). Moreover, computing $\zeta_\alpha$, the Hausdorff dimension of $K$ relative to $\nu_\alpha$ (resp. $\nu_{\alpha,1}$) amounts to computing the Euclidean Hausdorff dimension of $L_\alpha(K_\mu)$ (resp. $L_\alpha(K_{\mu_1}))$. Now we can use the fact that a.s. $L_\alpha(E)=\alpha\dim E$ for all subsets $E$ of $[0,1]$ \cite[III.5]{Bertoin96} to conclude that $\zeta_\alpha=\alpha \zeta$.
\end{proof}

One may observe that in some sense the effect of combining  the L\'evy process  $L_\alpha$ with the measure $\mu_1$ can be thought as a  'random snowflaking' (with exponent $\alpha$) of the random metric induced by $\nu_\alpha$, see \cite{He}.

\section{Multifractal nature of the measures $\mu_\beta$, $\beta\ge 1$}\label{se:multifractal}

The knowledge of the continuity of the measure $\mu_1$ has some consequences in the multifractal analysis of Mandelbrot measures and the multifractal analysis of L\'evy processes in multifractal time.
\medskip

Recall that given a positive Borel measure $\nu$ supported on a compact metric space $(X,d)$, its multifractal analysis consists in computing the Hausdorff dimension of the level sets of the pointwise H\"older exponent of $\nu$, namely the sets
$$
E_\nu(\gamma)=\Big\{ x\in X: \liminf_{r\to 0^+}\frac{\log \nu(B(x,r))}{\log (r)}=\gamma\Big \} \quad (\gamma\in [0,\infty]).
$$ 
Throughout, we adopt the convention that a set has a negative dimension if and only if it is empty. 

Now let $\mu$ (resp. $\mu_1$) be the subcritical (resp. critical) measure considered in  section \ref{subse:KPZagain} and  set  $\tau(s):=\phi(s)-1=-\widetilde\phi(s)$. Under suitable assumptions, the multifractal nature of the Mandelbrot measure $\mu$ has been studied in \cite{HoWa,Molchan96}, in which it is shown that  for each $\gamma$ such that $\tau^*(\gamma)>0$ one has almost surely  $\dim\, E_\mu(\gamma)=\tau^*(\gamma)$ (in these papers the pointwise H\"older exponent is associated with the dyadic intervals rather than centered balls). 
Here $\tau^*$ stands for the Legendre transform
$$
\tau^*(\gamma):=\inf_{t\in\R} \big(t\gamma-\tau(t)\big).
$$
This result is strengthened in \cite{Ba00}: almost surely  (simultaneously) for all $\gamma$ such  that $\tau^*(\gamma)>0$ it holds that $\dim\, E_\mu(\gamma)=\tau^*(\gamma)$.

Moreover, in \cite{Ba00} the question for the at most two values of $\gamma$ for which $\tau^*(\gamma)=0$ is partially solved. Here one employs the fact that in this case  $\gamma$ takes the form $\tau'(s_0)$, and one verifies that  $E_\mu(\tau'(s_0))$ carries a piece of  a critical measure of $\widetilde\mu_{1}$, where $\widetilde \mu_{1}$ is constructed by using instead the  normalized variable $e^{\widetilde \xi}:=\frac{e^{s_0\xi}}{2\E e^{s_0\xi}}$. In particular, this set  is nonempty, but its  Hausdorff dimension equals $0$. 
The fact that we now know that $\widetilde\mu_{1}$ is atomless makes it possible to strengthen the above result: $E_\mu(\tau'(s_0))$ is not countable. 

 The results of \cite{Ba00} hold also for the critical measures $\mu_1$. Then in \cite{BaJin10} in which limit of complex multiplicative cascades are studied, a complete answer was given for the multifractal behaviour of $\mu$, and
 because we now know that $\mu_1$ is atomless they extend easily  to $\mu_1$, and we may state the following result without proof. 
\begin{theorem}
Let $\nu=\mu$ or $\nu=\mu_1$ according to $\E \, e^{\xi} = \frac{1}{2}$ and $ \E \, \xi e^\xi <0$ or $\E \, e^{\xi} = \frac{1}{2}$ and $ \E \, \xi e^\xi =0$. Suppose that $\phi(s)>-\infty$ for all $s\in\R$ if $\nu=\mu$, and $\phi(s)>-\infty$ in a neighborhood of $(-\infty,1]$ if $\nu=\mu_1$. With probability 1, for all $\gamma\in[0,\infty]$, $E_\nu(\gamma)\neq\emptyset$ if and only if $\gamma$ belongs to the compact interval $I=\{\gamma:\tau^*(\gamma)=\inf_{s\in\R}(s\gamma-\tau(s))\ge 0\}$, and in this case $\dim E_\nu(\gamma)=\tau^*(\gamma)$. Moreover, $\min (I)=0$ if and only if $\nu=\mu_1$.  
\end{theorem} 
\noindent The previous results can be extended if $[0,1]$ is endowed with a random metric associated with a non-degenerate  Mandelbrot cascade (as it was done in [3]) or a critical Mandelbrot cascade built simultaneously with $\mu$ or $\mu_1$.

Given $\alpha\in (0,1)$, the multifractal nature of the measure $\nu_\alpha$ associated with $L_\alpha$ and $\mu$ as in Section \ref{subse:KPZagain} has been studied in \cite{BaSe07}.   The case of $\nu_{\alpha,1}$ was not treated in \cite{BaSe07}, mainly because of the lack of information on the discrete or continuous nature of $\mu_1$. After our Theorem \ref{th:elementary} it is not hard to adapt the approach developed in \cite{BaSe07} to achieve the multifractal analysis of $\nu_{\alpha,1}$. It is even easier than that of $\nu_\alpha$ because the difficult discussion associated with the degree of approximations of the points of $[0,1]$ by the atoms of $\nu_\alpha$ is not necessary. Consequently we just state the result:
\begin{theorem}
Let $\alpha\in (0,1)$. Suppose that $\phi(s)>-\infty$ on a neighborhood of $(-\infty,1)$. Let $\tau_\alpha(s)=\min (\tau(s/\alpha),0)$. Let $\nu=\nu_\alpha$ or $\nu=\nu_{\alpha,1}$ according to $\E \, e^{\xi} = \frac{1}{2}$ and $ \E \, \xi e^\xi <0$ or $\E \, e^{\xi} = \frac{1}{2}$ and $ \E \, \xi e^\xi =0$. With probability 1, for all $\gamma\in[0,\infty]$, $E_\nu(\gamma)\neq\emptyset$ if and only if $\gamma$ belongs to the compact interval $I=[0,\gamma_{\max}=\max\{\gamma:\tau_\alpha^*(\gamma)\ge 0\}]$, and in this case $\dim E_\nu(\gamma)=\tau_\alpha^*(\gamma)$. 
\end{theorem} 
In particular, when $\nu=\nu_\alpha$, $\tau^*$ is linear of slope $\alpha$ over $[0,\tau'(1)/\alpha]$ and strictly concave over $[\tau'(1)/\alpha, \gamma_{\max}]$, the linear part being reminiscent of the atoms of $L_\alpha$, while when $\nu=\nu_{\alpha,1}$, such a linear part disappears.

\section{Proofs of Theorem \ref{th:exact} and Lemma \ref{le:phimodulus}}\label{se:generating}

We will prove Theorem \ref{th:exact}  using the results   in \cite{we11}. Let $\beta >1$ and denote by $Z_n$  the random variable
\begin{equation}
\label{Zn}
Z_n=
\sum_{\sigma\in\Sigma_n} \left(e^{X_\sigma }Y_1^{(\sigma)}\right)^\beta.
\end{equation}
It will be convenient to define the following  reparametrization of  its Laplace transform:
\begin{equation}\label{eq:sophi}
H_{n,\beta}(x):=\E \ e^{-e^{\beta x}Z_n}. 
\end{equation}
The convergence of $c(n)n^{\beta/2}Z_n$  in distribution (where $\log c(n)$ is a bounded sequence) will follow as we prove that there is a  bounded sequence $C(n)$ so that 
\begin{equation}
\label{eq:Hnb}
H_{n,\beta}(x+\log \sqrt{n}+C(n))
\end{equation}
converges for all $x\in \R$ (to a function with limit 1 at  $-\infty$)  as $n\to\infty.$

We start by deriving a recursion relation for $H_{n,\beta}(x)$ in $n$. Given a nonnegative random variable $Y$ we
define  the random variable $T_\beta Y$ by
\begin{equation}
T_\beta Y\stackrel{d}{=} \left((e^{ \xi_0}Y^{(0)})^\beta +(e^{ \xi_1}Y^{(1)})^\beta
\right)^\frac{1}{\beta},
\end{equation}

\vspace{0.3cm}

\noindent where $Y^{(0)}$ and $Y^{(1)}$ are independent copies of $Y$ and $\xi_0$ and $\xi_1$ are independent copies of $\xi$ which are also independent of  $Y^{(0)}$ and $Y^{(1)}$. With this definition we have
\begin{equation}
Z_n\stackrel{d}{=}\left(T_\beta^n Y_1\right)^\beta.
\end{equation}
Passing to Laplace transforms  and using independence, we get the desired recursion:

\begin{align*}
H_{n+1,\beta}(x)&=\E\left(\exp\left(-e^{\beta x}(T_\beta^{n+1}Y_1)^\beta\right)\right)\\
&=\E\left(\exp\left(-e^{\beta x}\left((e^{ \xi_0} (T_\beta^n Y_1)^{(0)})^\beta+(e^{ \xi_1} (T_\beta^n Y_1)^{(1)})^\beta\right)\right)\right)\\
&=\E\left(\exp\left(-e^{\beta x} e^{\beta \xi} (T_\beta^n Y_1)^\beta\right)\right)^2\\
&=\left(\int \rho(y)H_{n,\beta}(x+y)dy\right)^2,
\end{align*}

\noindent where $\rho$ is the density of the distribution of $\xi$. 
$H_{n,\beta}(x)$ is determined from this recursion given the  initial data 
$$
H_{0,\beta}(x)=\E(\exp(-e^{\beta x} Y_1^\beta)).
$$
To make the connection to  \cite{we11} we restrict to the Gaussian case, i.e. $\xi\sim N(-2\log 2,2\log 2)$, and define 
$$G_{n,\beta}(x)=\left(H_{n,\beta}( {2n\log 2}-\sqrt{2\log 2}x)\right)^{\frac{1}{2}}.$$
Some calculation gives the following recursion
\begin{equation}
\label{eq:recursion}
G_{n+1,\beta}(x)=\int \frac{e^{-\frac{y^2}{2}}}{\sqrt{2\pi}}\big(G_{n,\beta}(x+y)\big)^2dy
\end{equation}
with initial data 
\begin{equation}
\label{eq:initial}
G_{0,\beta}(x)=\left(\E\ e^{-e^{-\beta \sqrt{2\log 2} x}Y_1^\beta}\right)^{\frac{1}{2}}.
\end{equation}
Thus finding a bounded sequence $C(n)$ such that 
 \eqref{eq:Hnb} converges amounts to finding a  bounded sequence $C'(n)$  so that
\begin{equation}
 \label{eq:seq}
 G_{n,\beta}\Big(x+n\sqrt{2\log 2}-\frac{\log {n}}{2\sqrt{2\log 2}}+C'(n)\Big)
 \end{equation} converges for all $x\in \R$ to an appropriate limit.
The  recursion \eqref{eq:recursion} was studied in \cite{we11}. The main result is

\begin{theorem}
\label{th:webb}
{\rm(a)}\quad Let $G_n^{(\alpha)}$ be given by the recursion \eqref{eq:recursion} with initial data 
\begin{equation}
\label{eq:initial2
}G_0^{(\alpha)}(x)=\exp(-e^{-\alpha x})\ \ 0<\alpha\leq\infty
\end{equation}
{\rm (}where $G_0^{(\infty)}$ is the Heaviside function $\theta(x)${\rm )}. Let 
$
m_n^{(\alpha)}=\left(G_n^{(\alpha)}\right)^{-1}\left(\frac{1}{2}\right)$. 
Then, (as $n\to\infty$)   $G_n^{(\alpha)}\left(x+m_n^{(\alpha)}\right)$ converges uniformly on $\R$ to a function $w^{(\alpha)}$ satisfying
\begin{equation}\label{eq:stationary}
w^{(\alpha)}(x)=\int \frac{e^{-\frac{y^2}{2}}}{\sqrt{2\pi}} w^{(\alpha)}(x+y+c(\alpha))^2dy,
\end{equation}
where 
\begin{equation}
c(\alpha)=
\left\lbrace\begin{array}{c}
\frac{\alpha}{2}+\frac{\log 2}{\alpha}, \ \mathrm{for} \ \alpha\leq\sqrt{2\log 2} \\
c(\sqrt{2\log 2}), \ \mathrm{for} \ \alpha>\sqrt{2\log 2}
\end{array}\right. .
\end{equation}
 The function $w^{(\alpha)}$ is the unique solution to this equation under the assumptions that $w^{(\alpha)}(0)=\frac{1}{2}$, $w^{(\alpha)}(-\infty)=0$, $w^{(\alpha)}(\infty)=1$ and $w^{(\alpha)}$ is increasing.
Moreover, the shift sequence $(m_n^{(\alpha)})$ exhibits  the following asymptotic behavior:
\begin{equation}
\label{eq:mn}
m_n^{(\alpha)}=\left\lbrace\begin{array}{c}
c(\alpha)n+\gamma(\alpha)+\mathit{o}(1), \ \mathrm{for} \ \alpha<\sqrt{2\log 2} \\
\sqrt{2\log 2}n-\frac{1}{2\sqrt{2\log 2}}\log n+\mathcal{O}(1), \ \mathrm{for} \ \alpha=\sqrt{2\log 2}\\
\sqrt{2\log 2}n-\frac{3}{2\sqrt{2\log 2}}\log n+\mathcal{O}(1), \ \mathrm{for} \ \alpha>\sqrt{2\log 2}
\end{array}\right. ,
\end{equation}
where $\gamma(\alpha)$ is a some real number.

\noindent {\rm(b)}\quad   Suppose\footnote{In this section the notation $A(x)\sim B(x)$
as $x\to \infty$ is a shorthand for $\lim_{x\to\infty} \big(A(x)/B(x)\big)=c$ for some $c\in (0,\infty ).$} the initial condition
 $G_0(x)$ is an  increasing function with  $G_0(-\infty )=0$ and 
 $1-G_0(x)\sim e^{-\alpha x}$ as $x\to\infty$.  If  $\alpha<\sqrt{2\log 2}$
 the results of (a)  hold.  Moreover,  for  $\alpha\geq \sqrt{2\log 2}$,
if  the sequence $G_n^{(\alpha)}\left(x+m_n^{(\alpha)}\right)$ (defined now using the initial data $G_0$) converges uniformly then \eqref{eq:mn} holds.
\end{theorem}
\begin{proof} 
For (a) see Theorem 3.4, Corollary 3.5 and sections 4 and 5 in  \cite{we11}. For (b) $\alpha<\sqrt{2\log 2}$,  see Theorem 3.4 and Lemma 5.3 in \cite{we11} and for  $\alpha\geq \sqrt{2\log 2}$  see the argument in Section 6  of \cite{we11}.
\end{proof} 

To apply this Theorem we need the asymptotics of the initial data $G_{0,\beta}$ in  \eqref{eq:initial}. In fact, we claim that
\begin{equation}
\label{eq:initial1}
1-G_{0,\beta}(x)\sim e^{-\sqrt{2\log 2} x}\ \ {\rm as}\ \ x\rightarrow \infty.
\end{equation}
Indeed, 
let $\phi_\beta(t)=\E(\exp(-t Y_1^\beta))$. Using Theorem \ref{th:tail} we see that as $t\rightarrow 0$
\begin{eqnarray}
\label{eq:asym}
1-\phi_\beta(t)&=&t\int_0^\infty e^{-tx}\Prob(Y_1^\beta >x)dx\\
&=&\int_0^\infty e^{-x}\Prob(Y_1 >t^{-\frac{1}{\beta}}x^\frac{1}{\beta})dx
\sim t^{\frac{1}{\beta}}.
\end{eqnarray}
Hence $1-H_{0,\beta}(x)\sim (e^{\beta x})^{\frac{1}{\beta}}=e^x$ as $x\rightarrow -\infty$ which
translates to \eqref{eq:initial1} for $G_{0,\beta}(x)$ 

Thus the convergence of \eqref{eq:seq} follows from Theorem \ref{th:webb} (b)
provided we show the sequence $G_{n,\beta}(x+m_{n,\beta})$ converges
with the choice  $m_{n,\beta}\equiv \left(G_{n,\beta}\right)^{-1}\left(\frac{1}{2}\right)$.
We prove the convergence of  $G_{n,\beta}(x+m_{n,\beta})$  by comparing it to the solutions provided by  Theorem \ref{th:webb} (a)
with $\alpha<\sqrt{2\log 2}$ and $\alpha=\infty$ using the following maximum principle
from  \cite{we11}:
\begin{proposition}
 \label{max}Let $G_n^1$ and $G_n^2$ be given by the recursion \eqref{eq:recursion} with initial data $G_0^1$ and $G_0^2$ with the property that there exists a point $x_0$ so that $G_0^2(x)\geq G_0^1(x)$ for $x\geq x_0$ and $G_0^2(x)\leq G_0^1(x)$ for $x\leq x_0$. Then for every $n\geq 1$ there exists a point $x_n$ so that $G_n^2(x)\geq G_n^1(x)$ for $x\geq x_n$ and $G_n^2(x)\leq G_n^1(x)$ for $x\leq x_n$.
 The claim  holds also with strict inequalities.
\end{proposition}

We perform the comparison by considering the following family of initial conditions
with $\beta_1>0$:
\begin{equation}
G_0^{(\beta_1,\beta)}(x)=\left(\E\ e^{-e^{-\beta_1 \sqrt{2\log 2}x} Y_1^{\beta}}\right)^{\frac{1}{2}}.
\end{equation}
Note that $G_0^{(\beta,\beta)}=G_{0,\beta}$,  $G_0^{(\infty,\beta)}=\theta(x)$
and from \eqref{eq:asym} we have
\begin{equation}
\label{eq:initial3}
1-G_0^{(\beta_1,\beta)}(x)\sim e^{-\frac{\beta_1}{\beta}\sqrt{2\log 2} x}\ \ {\rm as}\ \ x\rightarrow \infty.
\end{equation}
Let $G_n^{(\beta_1,\beta)}$ be given by the recursion \eqref{eq:recursion} with this initial data and $m_{(\beta_1,\beta),n}=\left(G_n^{(\beta_1,\beta)}\right)^{-1}\left(\frac{1}{2}\right)$ (one can use the recursion to check that $G_n^{(\beta_1,\beta)}$ is strictly increasing and this is well defined). 

\begin{lemma} 
\label{compa}
For $x\geq 0$,  $G_n^{(\beta_1,\beta)}(x+m_{(\beta_1,\beta),n})$ is increasing in $\beta_1$ and for $x\leq 0$ it is decreasing in $\beta_1$.
\end{lemma}

\begin{proof} Let $\beta_1> \beta_1'$ and $n_0$ be some fixed positive integer. Define $G_0^2(x)= G_0^{(\beta_1,\beta)}(x+m_{(\beta_1,\beta),n_0})$ and $G_0^1(x)=G_0^{(\beta_1',\beta)}(x+m_{(\beta_1',\beta),n_0})$. We then note that if

\begin{equation}
x> x_0:=\frac{\beta_1'm_{(\beta_1',\beta),n_0}-\beta_1m_{(\beta_1,\beta),n_0}}{\beta_1-\beta_1'},
\end{equation}
then

\begin{equation}
\exp\left(-e^{-\beta_1 \sqrt{2\log 2}(x+m_{(\beta_1,\beta),n_0})} Y_1^{\beta}\right)>\exp\left(-e^{-\beta_1' \sqrt{2\log 2}(x+m_{(\beta_1',\beta),n_0})} Y_1^{\beta}\right)
\end{equation}
and $G_0^2(x)> G_0^1(x)$. Similarly for $x< x_0$, $G_0^2(x)< G_0^1(x)$. Thus by 
Proposition \ref{max}  there exists a point $x_n$ so that $G_n^2(x)> G_n^1(x)$ for $x>x_n$, and the opposite inequality holds for $x<x_n$. Let us set $n=n_0$ and note that $G_{n_0}^2(0)=G_{n_0}^1(0)=\frac{1}{2}$. Thus $x_{n_0}=0$. Since $n_0$ was arbitrary, we have proven our claim. 
\end{proof}
We can now finish the proof of convergence of the
sequence $G_{n,\beta}(x+m_{n,\beta})$. 
First,  by Theorem \ref{th:webb}(b)  and \eqref{eq:initial3}. the quantity $G_n^1(x):=G_n^{(\beta_1,\beta)}(x+m_{(\beta_1,\beta),n})$ 
converges to $w^{(\alpha)}(x)$ uniformly in $x$ with $\alpha=\frac{\beta_1}{\beta}\sqrt{2\log 2}$ provided
$\beta_1<\beta$. Second, by Theorem \ref{th:webb} (a) $G_n^2(x):=G_n^{(\infty,\beta)}(x+m_{(\infty,\beta),n})$ 
converges to $w^{(\infty)}(x)=w^{(\sqrt{2\log 2})}(x)$,
uniformly in $x$. Finally, by Lemma \ref{compa} we have
$$
G_n^1(x)<G_{n,\beta}(x+m_{n,\beta})<G_n^2(x)
$$
for $x>0$, and the opposite inequalities hold for  $x<0$. Since $w^{(\alpha)}(x)\to w^{(\sqrt{2\log 2})}(x)$
as $\alpha\uparrow  \sqrt{2\log 2}$, uniformly in $x$,  we conclude that the
sequence $G_{n,\beta}(x+m_{n,\beta})$ is convergent and hence by Theorem \ref{th:webb} (b) 
that $m_{n,\beta}\equiv \left(G_{n,\beta}\right)^{-1}\left(\frac{1}{2}\right)$ is given by \eqref{eq:mn}.
Hence we have shown the existence of a bounded sequence $c(n)$ such that when the left hand side of  \eqref{Zn} is multiplied by $c(n)$, the product converges
 in distribution  to some non-trivial random variable $Z$. As the $Z_n$:s satisfy the 'smoothing recursion'
$$
Z_{n+1}\stackrel{d}{=} e^{\beta \xi_0}Z_n^{(0)}+e^{\beta \xi_1}Z_n^{(1)},
$$
it follows that after normalization $Z$ is a fixed point of the smoothing transform equation
$$
Z\stackrel{d}{=} e^{\beta \xi_0}Z^{(0)}+e^{\beta \xi_1}Z^{(1)}.
$$
This has (up to a constant factor)the unique solution $Y_\beta$  \cite{duli83} .
The proof of Theorem \ref{th:exact} is complete.

\vskip 2mm

We end this Section by proving Lemma \ref{le:phimodulus} and proving a similar result needed in the last section.

\begin{proof}[Proof of Lemma \ref{le:phimodulus}] Recall the partition function
\eqref{eq:defZbeta} and set
$$
K_{n,\beta}(x):=\E \ e^{-e^{\beta x}Z_{\beta,n}}. 
$$
Thus
$$
\phi_{\beta,n}({e^{\beta x}})=K_{n,\beta}(x+{\alpha_\beta}\log n)
$$
with $\alpha_1=1/2$ and  $\alpha_\beta=3/2$ for $\beta>1$. Proceeding as
earlier we get
$$
K_{n,\beta}(x)=\tilde H_{n,\beta}(x)
$$
where $\tilde H_{n,\beta}$ solves the same recursion as $H_{n,\beta}$, but  with initial
condition
$$
\tilde H_{0,\beta}(x)=\exp(-e^{\beta x}).
$$
In the``$G$-language" this becomes  
$$\tilde G_{0,\beta}(x)=\exp(-\frac{1}{2}e^{-\beta \sqrt{2\log 2}x})=G_0^{(\beta\sqrt{2\log 2})}\Big(x+\frac{\log 2}{\beta\sqrt{2\log 2}}\Big)$$
 and
$$
\tilde G_{n,\beta}(x)=G_n^{(\beta\sqrt{2\log 2} )}\Big(x+\frac{\log 2}{\beta\sqrt{2\log 2}}\Big).
$$
Hence 
\begin{align}
\label{phig}
\phi_{\beta,n}({e^{\beta x}})&=\left(\tilde G_{n,\beta}\Big(-\frac{x}{\sqrt{2\log 2}}+\sqrt{2\log 2}n-\frac{\alpha_\beta}{\sqrt{2\log 2}}\log n\Big)\right)^2\\
&=\left(G_n^{(\beta\sqrt{2\log 2}) }\Big(-
\frac{x}{\sqrt{2\log 2}}+m_n^{(\beta\sqrt{2\log 2})}+a_n\Big)\right)^2\nonumber
\end{align}
for some bounded sequence $a_n$ (which depends on $\beta$). 

As in Lemma \ref{compa} we get 
$$G_n^{(\alpha)}\left(x+m_n^{(\alpha)}\right)<G_n^{(\alpha')}\left(x+m_n^{(\alpha')}\right)$$
if $\alpha<\alpha'$ and $x>0$. 
Combining this inequality with eq. \eqref{phig} allows us to get the $t\to 0$ asymptotics of the low temperature object $\phi_{\beta,n}(t)$ 
from the corresponding   asymptotics in high temperature. Indeed, fix $\beta'<1$. Then for $x$ small enough and for
 some  bounded sequences $b_n, c_n$   for all $n$ we get
$$
\phi_{\beta,n}({e^{\beta x}})\geq \left(G_n^{(\beta'\sqrt{2\log 2}) }\Big(-
\frac{x}{\sqrt{2\log 2}}+m_n^{(\beta'\sqrt{2\log 2})}+b_n\Big)\right)^2=\E \ e^{-c_ne^{\beta' x}\frac{Z_{\beta',n}}{\E Z_{\beta',n}}}
$$
where $Z_{\beta',n}$ is the high temperature partition function. Using $e^{-x}\geq 1-x$ for $x\geq 0$ we get 
$$
\E \ e^{-c_ne^{\beta' x}\frac{Z_{\beta',n}}{\E Z_{\beta',n}}}\geq 1-c_ne^{\beta' x}
$$
and therefore, by denoting $c=\sup_{n\geq 1} c_n$ we have
$$
1-\phi_{\beta,n}(t)\leq c t^{\beta'/\beta}.
$$
Thus, fixing $\gamma<1/\beta$ it holds that
$$
1-\phi_{\beta,n}(t)\leq C(\beta,\gamma)t^{\gamma}.
$$
for some  $C(\beta,\gamma)<\infty$ and $t\leq t(\beta,\gamma)$ with $ t(\beta,\gamma)>0$. Since this inequality is trivial for $t$ bounded away from zero the claim follows.
\end{proof}

\begin{lemma}
\label{le:tailpr}
For any $\beta>1$, $\theta>0$ and $q\in(0,\frac{1}{\beta})$, 

\begin{equation}
\mathbb{P}\left(\sum_{\sigma\in \Sigma_n}(\sqrt{n}\mu_1(I_\sigma))^\beta> n^\theta\right)\leq C(q)n^{-q\theta}.
\end{equation}
\end{lemma}

\begin{proof}
The proof is almost identical to that of Lemma \ref{le:phimodulus}. We recall that we argued at the beginning of the previous page that $G_{n,\beta}(x+m_{n,\beta})$ converges uniformly. Moreover, we know by \eqref{eq:initial3} that $1-G_{0,\beta}(x)\sim e^{-\sqrt{2\log 2}x}$. Thus by part (b) of Theorem \ref{th:webb}, $m_{n,\beta}=\sqrt{2\log 2}n-\frac{1}{2\sqrt{2\log 2}}\log n+\mathcal{O}(1)$. Lemma \ref{compa} then implies that for small enough $x$ and any $\beta_1<\beta$

\begin{align*}
H_{n,\beta}\left(x+\frac{1}{2}\log n\right)&=G_{n,\beta}\left(-\frac{x}{\sqrt{2\log 2}}+\sqrt{2\log 2}n-\frac{1}{2\sqrt{2\log 2}}\log n\right)^2\\
&=G_{n,\beta}\left(-\frac{x}{\sqrt{2\log 2}}+m_{\beta,n}+\mathcal{O}(1)\right)^2\\
&\geq G_n^{(\beta_1,\beta)}\left(-\frac{x}{\sqrt{2\log 2}}+m_{(\beta_1,\beta),n}+\mathcal{O}(1)\right)^2.
\end{align*}

\vspace{0.3cm}

It follows from the proof of Theorem 3.4 and Lemma 5.3 of \cite{we11} that for $x\geq 0$, $1-G_n^{(\beta_1,\beta)}(x+m_{(\beta_1,\beta),n})\sim e^{-\frac{\beta_1}{\beta} \sqrt{2\log 2}x}$. Thus we conclude that there exists a constant $C>0$ so that for small enough $x$ (say $x\leq -M$)

\begin{equation}
1-H_{n,\beta}\left(x+\frac{1}{2}\log n\right)\leq C e^{\frac{\beta_1}{\beta}x}.
\end{equation}

\vspace{0.3cm}

We then have by Markov's inequality for any $q>0$

\begin{align*}
\mathbb{P}&\left(\sum_{\sigma\in \Sigma_n}(\sqrt{n}\mu_1(I_\sigma))^\beta> n^\theta\right)\leq n^{-q\theta}\E \left(\left(\sum_{\sigma\in \Sigma_n}(\sqrt{n}\mu_1(I_\sigma))^{\beta}\right)^q\right)\\
&=\widetilde{C}(q)n^{-\theta q}\int_0^\infty t^{-1-q}\left(1-H_{n,\beta}\left(\frac{1}{\beta}\log t+\frac{1}{2}\log n\right)\right)dt\\
&\leq \widehat{C}(q)n^{-\theta q} \left(\int_0^{e^{-\beta M}} t^{-1-q}  t^{\frac{\beta_1}{\beta^2}}dt+\int_{e^{-\beta M}}^\infty t^{-1-q}dt\right).
\end{align*}

\vspace{0.3cm}

We see that for $q\in(0,\frac{1}{\beta})$ both of the integrals are finite (when we choose $\beta_1$ close enough to $\beta$). Thus we find our claim.

\end{proof}

\section{Complement on $\mu_1$-almost everywhere local behavior of $\mu_1$}\label{se:localbehavior}

In the subcritical case $\beta < 1$ there exist very good bounds for the almost sure fluctuations of the measure $\mu_\beta$ considered at $\mu_\beta$-almost every $x \in [0,1]$. Denoting by $I_n(x)$ the unique half-open dyadic interval of level $n$ containing $x$, under rather general conditions on $\xi$ it holds that almost surely for $\mu_\beta$-almost every $x \in [0,1]$
$$
2^{-\alpha n} e^{-b \sqrt{n \log \log n}} \leq \mu_\beta(I_n(x)) \leq 2^{-\alpha n} e^{b \sqrt{n \log \log n}}
$$
for all large $n$, where $\alpha$ and $b$ are constants depending on $\beta$ and the distribution of $\xi$; see \cite{liu01} for the precise statement of the result. In effect, the measure $\mu_\beta$ satisfies a kind of a law of the iterated logarithm.

In this section we consider the corresponding fluctuation problem in the critical case $\beta = 1$, i.e. the question of finding bounds $\psi(n)$, $\phi(n)$ such that almost surely, for $\mu_1$-almost every $x \in [0,1]$ one has
$$
\psi(n) \leq \mu_1(I_n(x)) \leq \phi(n)
$$
for all large $n$. Clearly, the optimal fluctuation bounds cannot have the same form as in the subcritical case, as one would need to have $\alpha = 0$ above. Our method of obtaining bounds depends on Theorem \ref{th:exact} and thus we restrict to the case of a Gaussian $\xi$.

\begin{theorem}\label{thm:gauge2} Suppose $\xi$ is Gaussian. Then the following statements hold.
\begin{enumerate} 
\item Let $f:\N_+\to \R_+^*$ be a nonincreasing function converging to 0 at infinity. If $\displaystyle\liminf_{n\to\infty} \frac{\log f(n)}{-\sqrt{n\log(n)}}>\sqrt{2 \log 2}$ then almost surely, 
$$
\mu_1\left( \left\{ x: \mu_1(I_n(x)) \ge f(n)\textrm{ for infinitely many n } \right\} \right) = \mu_1([0,1]).
$$
\item Let $\displaystyle 
f_\alpha(n) = \exp\left( - \sqrt{6 \log 2} \sqrt{n \left(\log n +\alpha  \log \log n \right)}\right)
$ for $\alpha > \frac{1}{3}$. Then almost surely,
$$
\mu_1\left( \left\{ x: \mu_1(I_n(x)) \geq f_\alpha(n) \textrm{ for all but finitely many }n \right\} \right) = \mu_1([0,1]).
$$
\item Almost surely, for all $k\in\N$
$$
\mu_1\left( \left\{ x: \mu_1(I_n(x)) \le n^{-k} \textrm{ for all but finitely many n } \right\} \right) = \mu_1([0,1]).
$$
\end{enumerate}
\end{theorem}
\begin{proof}
We start with the proofs of (1) and (2) that can be achieved by the application of general moment estimates. We remark that these statements have analogues that can be proven by the same method for general $\xi$. Let $f: \N \to \R^+$ be an ultimately nonincreasing function tending to $0$ at infinity. We consider the $\mu_1$-measures of the sets
$$
E_n^f = \{ x: \mu_1(I_n(x)) \leq f(n) \}.
$$
Let $(\eta_n)_{n\ge 1}$ be a sequence taking values in $(0,1)$ and write
\begin{align*}
\mu_1(E_n^f) & = \int_0^1 \ind_{\{\mu_1(I_n(x)) \leq f(n)\}} \d \mu_1(x)
= \sum_{\sigma \in \Sigma_n} \mu_1(I_\sigma) \ind_{\{\mu_1(I_\sigma) \leq f(n)\}} \\
& \leq \sum_{\sigma \in \Sigma_n} \mu_1(I_\sigma) \left (\frac{f(n)}{\mu_1(I_\sigma)}\right)^{\eta_n}
= \sum_{\sigma \in \Sigma_n} \mu_1(I_\sigma)^{1-\eta_n} f(n)^{\eta_n}.
\end{align*}
By the characterization \eqref{eq:semi-stable} of the law of $\mu_1$, we have
$$
\E \, \mu_1(E_n^f) \leq f(n)^{\eta_n} \sum_{\sigma \in \Sigma_n} \E e^{(1-\eta_n)X_\sigma} \, \E Y_1^{1-\eta_n} = f(n)^{\eta_n} 2^{n \eta_n^2} \E Y_1^{1-\eta_n}.
$$
Theorem \ref{th:tail} implies the existence of a constant $C > 0$ such that $\E Y_1^{1-\eta_n} \leq C/\eta_n$, which gives
\begin{equation}\label{eq:gauge-ld-estimate}
\E \, \mu_1(E_n^f) \leq C \exp\left( \eta_n^2 n \log 2 + \eta_n \log f(n) - \log \eta_n \right).
\end{equation}
By solving for the zero of the derivative, the expression in the exponential is minimized for $\eta_n > 0$ by the choice
$$
\eta_n^{\min} = \frac{-\log f(n)}{4n \log 2} + \frac{\sqrt{8n \log 2 + (\log f(n))^2}}{4n \log 2}.
$$
To get more manageable expressions, we choose $\eta_n = \frac{-\log f(n)}{2 n \log 2} < \eta_n^{\min}$ to get the estimate
$$
\E \, \mu_1(E_n^f) \leq C \exp\left( - \frac{(\log f(n))^2}{4 n \log 2} + \log n - \log (-\log f(n)) + \log (2 \log 2) \right).
$$

Under the assumption of part (1) of the theorem, for some $\varepsilon > 0$ there exists a sequence $(n_k)_{k \geq 1}$ of indices such that
$$
-\log f(n_k) \geq (\sqrt{2 \log 2}+\varepsilon) \sqrt{n_k \log n_k}
$$
for all $k \geq 1$. Thus
$$
\E \, \mu_1(E_{n_k}^f) \leq C' \exp\left( -\frac{(\sqrt{2 \log 2}+\varepsilon)^2 \log n_k}{4 \log 2} + \frac{1}{2} \log n_k - \frac{1}{2} \log \log n_k \right),
$$
which shows that $\E \, \mu_1(E_{n_k}^f) \to 0$ as $k \to \infty$. We may thus extract a subsequence of $(n_k)$, for convenience still denoted by $(n_k)$, for which
$$
\sum_{k \geq 1} \E \, \mu_1(E_{n_k}^f) < \infty,
$$
implying that $\sum_{k \geq 1} \mu_1(E_{n_k}^f) < \infty$ almost surely. An application of the Borel--Cantelli lemma to the measure $\mu_1$ allows us to conclude that almost surely the set
$$
\big\{ x \in [0,1]: \mu_1(I_{n_k}(x)) \leq f(n_k) \textrm{ for all but finitely many } k \big\}
$$
has $\mu_1$-measure $0$. This implies the claim.

To prove (2), let $f_\alpha$ be as in the statement. For $f = f_\alpha$ our earlier choice of $\eta_n$ is explicitly
$$
\eta_n = \frac{\sqrt{3} \sqrt{\log n + \alpha \log \log n}}{\sqrt{2 n \log 2}},
$$
which we plug into \eqref{eq:gauge-ld-estimate} to get
\begin{align*}
\E \, \mu_1(E_n^{f_\alpha}) & \leq C \exp\left( -\frac{3}{2} (\log n + \alpha \log \log n) - \log \frac{\sqrt{3} \sqrt{\log n + \alpha \log \log n}}{\sqrt{2 n \log 2}} \right) \\
& \leq C' \exp\left( -\log n - \frac{3 \alpha + 1}{2} \log \log n \right).
\end{align*}
We see that for $\alpha > 1/3$
$$
\sum_{n \geq 1} \E \, \mu_1(E_n^{f_\alpha}) < \infty,
$$
which implies that almost surely $\sum_{n \geq 1} \mu_1(E_n^{f_\alpha}) < \infty$. The claim now follows from the Borel--Cantelli lemma.

\vspace{0.3cm}

\noindent The proof of part (3) requires the use of subtler properties of the cascade. We will prove by induction that for all $k\in\N_+$ the following property $\mathcal{P}_k$ holds.
\begin{description}
\item[$\mathcal{P}_k$] For all $\gamma<k/2$, almost surely $\mu_1$-almost everywhere for $n$ large enough, one has $\mu_1(I_n(x))\le n^{-\gamma}$.
\end{description}
Notice that by Theorem \ref{th:modulus} this property holds for $k=1$.

Suppose $\mathcal{P}_k$ holds for some $k\in\N_+$. Fix $1/2<\gamma<(k+1)/2$ and let $\varepsilon \in (k+1-2\gamma, k)$. For each $N \geq 1$ let 
$$
E_N = \{x\in [0,1]: \,\forall\, n \geq N,\  \mu_1(I_n(x))\le n^{-(\gamma-1/2)}\}
$$ 
and note that from the assumption that $\mathcal{P}_k$ holds it follows that
$$
\mu_1\left(\cup_{N \geq 1} E_N\right) = \mu_1([0,1]).
$$
Setting $f(n)=n^{\varepsilon/2-(k+1)/2}$ we have, for all $n \geq N$ and $\beta > 1$,
\begin{align*}
\sum_{\sigma\in\Sigma_n: I_\sigma\cap E_N \neq \emptyset} \mu_1(I_\sigma) \mathbf{1}_{\{\mu_1(I_\sigma) \geq f(n)\}}
& \leq n^{-(\gamma-1/2)} \#\{\sigma\in\Sigma_{n}: \mu_1(I_\sigma) \geq f(n)\} \\
& \leq n^{-(\gamma-1/2)}  \sum_{\sigma\in\Sigma_{n}}\mu_1(I_\sigma)^\beta f(n)^{-\beta} \\
& = n^{-(\gamma-1/2)}  n^{-\beta(\varepsilon/2-(k+1)/2)} \sum_{\sigma\in\Sigma_{n}}\mu_1(I_\sigma)^\beta \\
&= n^{-\theta} \sum_{\sigma\in\Sigma_n}\big (n^{1/2}\mu_1(I_\sigma)\big )^\beta,
\end{align*} 
where $\theta=\gamma-\frac{1}{2}-\beta\frac{k-\varepsilon}{2}$. By our choice of $\varepsilon$ we may choose $\beta > 1$ so that $\beta < \frac{2\gamma - 1}{k - \varepsilon}$, which implies $\theta > 0$.
Now recall that by Theorem~\ref{th:exact}, $c(n)\sum_{\sigma\in\Sigma_{n}}\big (n^{1/2}\mu_1(I_\sigma)\big )^\beta$ converges in law to $Y_\beta$ for some bounded sequence $c(n)$, and moreover by Lemma \ref{le:tailpr}, for $q \in (0, 1/\beta)$ we have the uniform estimate $\Prob\left( \sum_{\sigma\in\Sigma_n} \big(n^\frac{1}{2} \mu_1(I_\sigma)\big)^\beta > n^{\theta/2} \right) \leq C(q) n^{-\theta q/2}$. Consequently, there exists an integer $\ell > 2/\theta$ such that for the sequence $(n_j)_{j=1}^\infty = (j^\ell)_{j=1}^\infty$ we have, almost surely for all $j$ large enough,
$$
\sum_{\sigma\in\Sigma_{n_j}}\big (n_j^{1/2}\mu_1(I_\sigma)\big )^\beta\le n_j^{\theta/2}
$$
and hence for $j$ large enough
$$
\mu_1\big (E_N\cap\{x: \mu_1(I_{j^\ell}(x))\ge f(j^\ell)\}\big ) \leq \sum_{\sigma\in\Sigma_{j^\ell}: I_\sigma\cap E_N\neq\emptyset} \mu_1(I_\sigma)\mathbf{1}_{\{\mu_1(I_\sigma) \geq f(j^\ell)\}} \leq j^{-\ell \theta/2}.
$$
It follows that almost surely, for all $N \geq 1$,
$$
\sum_{j^\ell \geq N} \mu_1\big( E_N \cap \{ x: \mu_1(I_{j^\ell}(x)) \geq f(j^\ell) \} \big) < \infty.
$$
By the Borel--Cantelli lemma, almost surely $\mu_1$-almost everywhere on $E_N$ we have $\mu_1(I_{j^\ell}(x)) \leq f(j^\ell)$ for $j$ large enough. But if $j^\ell < n \leq (j+1)^\ell$, we then also have
$$
\mu_1(I_n(x)) \leq \mu_1(I_{j^\ell}(x)) \leq f(j^\ell) = (j^\ell)^{\frac{\varepsilon}{2}-\frac{k+1}{2}} \leq 2^{\ell \frac{k+1-\varepsilon}{2}} n^{\frac{\varepsilon}{2}-\frac{k+1}{2}},
$$
and therefore we conclude that there exists a constant $C=C(\ell,k,\varepsilon)>0$ such that almost surely $\mu_1$-almost everywhere on $E_N$ we have
$$
\mu_1(I_n(x)) \leq C f(n) = C n^{-\frac{k+1-\varepsilon}{2}}
$$
for all $n$ large enough. By our assumption we have $\mu_1(\cup_{N \geq 1} E_N) = \mu_1([0,1])$, so we have shown that the desired conclusion holds for all $\gamma' < (k+1-\varepsilon)/2$. Since $\gamma$ can be taken arbitrarily close to $(k+1)/2$ and hence $\varepsilon$ arbitrarily close to 0, we are done.
\end{proof}

\end{document}